\documentclass[11pt]{amsart}

\usepackage{amsmath,amssymb,amscd}
\usepackage[mathscr]{eucal}
\newcommand{\Spec}{{\mathrm{Spec}\, }}

\newcommand{\bbar}[1]{{\overline{#1}}}
\newcommand{\cA}{{\mathcal A}}
\newcommand{\cB}{{\mathscr B}}
\newcommand{\cC}{{\mathscr C}}

\newcommand{\cE}{{\mathcal E}}

\newcommand{\cF}{{\mathcal F}}

\newcommand{\cG}{{\mathcal G}}
\newcommand{\cI}{{\mathcal I}}
\newcommand{\cJ}{{\mathcal J}}

\newcommand{\cO}{{\mathcal O}}

\newcommand{\cS}{{\mathscr S}}
\newcommand{\cT}{{\mathscr T}}
\newcommand{\uT}{\underline{\mathscr T}}
\newcommand{\uL}{\underline{L}}
\newcommand{\uP}{\underline{P}}

\newcommand{\cV}{{\mathcal V}}

\renewcommand{\sp}[1]{{\mathrm{Spec}}}

\renewcommand{\o}{\omega}

\newcommand{\sg}{\mathrm{Sing}}
\newcommand{\ma}{\mathrm{Max}}
\renewcommand{\max}{\mathrm{max}}

\theoremstyle{plain}
\newtheorem{thm}{Theorem}[section]

\newtheorem{pro}[thm]{Proposition}
\newtheorem{cor}[thm]{Corollary}
\theoremstyle{definition}
\newtheorem{Def}[thm]{Definition}
\newtheorem{rem}[thm]{Remark}
\newtheorem{exa}[thm]{Example}
\newtheorem{voi}[thm]{}

\begin{document}

\bigskip 
\title{Resolution algorithms and deformations}
\author{Augusto Nobile}
\begin{abstract}
An algorithm for resolution of singularities in characteristic zero  is described. It is expressed in terms of multi-ideals, that essentially are defined as a  finite sequence of pairs, each one consiting of  a sheaf of ideals and a positive integer. This approach is particularly simple and, as indicated by some results shown here, it seems suitable for applications to a good theory of simultaneous algorithmic resolution of singularities, specially for families parametrized by the spectrum of an artinian ring.
\end{abstract}

\address{Louisiana State University \\
Department of Mathematics \\
Baton Rouge, LA 70803, USA}



\email{nobile@math.lsu.edu}
\maketitle

\section*{Introduction}
\label{S:intro}

Algorithmic, constructive, or canonical methods to resolve singularities of algebraic varieties  attempt to clarify and simplify the original proof of the main desingularization theorem (\cite{Hir}). These are programs to eliminate the singularities of an algebraic variety by means of a sequence of blowing-ups with well determined regular centers. So far these algorithms proceed indirectly working primarily with some auxilary objects, such as marked ideals, basic objects or presentations. Suitable resolution algorithms for such objects imply similar results for varieties. 
  At present there are several algorithmic processes to desingularize algebraic varieties over fields of characteristic zero 
(\cite{BEV}, \cite{BM}, \cite{BMF}, \cite{Cu}, \cite{EH}, \cite{EV}, 
  \cite{Kol}, \cite{T2}, \cite{V1}, \cite{W}). 

Algorithmic resolutions (in characteristic zero) being available nowadays, it becomes reasonable to investigate   the possibility to  simultaneously resolve the members of a family of varieties, or their related objects, using a given  resolution algorithm. Results in this direction were obtained in  \cite{EN}, in the case where the parameter scheme is regular. In  \cite {T1} and \cite{T3}, the general case was studied using essentially the algorithm of \cite{EV}.

More precisely,   in \cite{T1} we considered the crucial 
 case where the parameter space is  the spectrum of an artinian ring, i.e.,  that of an infinitesimal deformation of an object over a field. 
 Most of the discussion of that paper 
 is in the context of {\it basic objects}, i.e., systems $(W,I,b,E)$ where $W$ is a variety smooth over a characteristic zero field $k$, $I$ is a coherent sheaf of ${\cO}_W$-ideals, $b$ a positive integer, and $E$ a set of regular divisors of $W$ with normal crossings. To develop a reasonable theory  of simultaneous resolution, or ``equiresolution'',  
  we try to imitate 
 what the  algorithm does when the base is  a  field. First,   we introduce basic objects over an artinian ring and a notion of permissible centers in this context, i.e., the centers we allow in our blow-ups. 
  Given  a basic object $B$  over an artinian ring,  
 we have a naturally defined closed fiber $B^{(0)}$, which is a basic object over a field. 
 Then we attempt to ``naturally extend'' the permissible 
centers 
 used in the  algorithmic resolution of $B^{(0)}$ to permissible centers of  $B$ and its transforms. 
When this can be done for  all the centers used in the algorithmic resolution of $B^{(0)}$,  
we say that $B$  is {\it algorithmically equisolvable}.

Working over a field,  the algorithmic resolution process usually requires an ``inductive step''. 
 Indeed, given a a basic object $B=(W,I,b,E)$ satisfying certain conditions, often it is necessary 
 to substitute  $B$ 
(near a point  $x \in W$) by another basic object $B^{\star}=(Z,J,c,E^{\star})$, where $Z$ is a suitable hypersurface, defined on an appropriate neighborhood $U$ of $x$. Since  $\dim Z < \dim W$,   by induction on the dimension, we have a  first (or zeroth) algorithmic resolution center for $B^{\star}$. This is a closed subscheme of $U$, i.e., a locally closed subscheme of $W$. Since these centers are defined just locally,  there is a glueing problem. But it can be proved  that they agree on intersections to produce a closed subscheme $C$ of $W$, which is a $B$-permissible center. This is the first algorithmic center for $B$.  

Working over an artinian ring $A$, in  \cite{T1} we have shown that it is possible,    to some extent, to imitate  the above constructions. But there is a drawback. Indeed, if  $B$ is a basic object over   $A$ 
 with special fiber $B^{(0)}$, we can impose reasonable conditions so that when $B^{(0)}$ is in the inductive situation, an analog for $B^{\star}$ is defined. 
 If, by induction on the dimension, $B^{\star}$ is algorithmically equisolvable, then 
we have a  first algorithmic equiresolution center $ C$ for $B^{\star}$. This   is a closed subscheme of $W$, which should be the first center for $B$. But, unfortunately, sometimes this subscheme $C$ of $W$ might not be a permissible center for $B$, because an equality of orders of certain ideals required for such centers may fail. Thus permissibility is a condition to be imposed in the definition of equiresolution, so this notion is not strictly recursive.

This ``pathology'' is due to the fact that  
 working over artinian rings, the local rings of our schemes  have nilpotents. The algorithm used in \cite{T1} involves certain constructions like   the {\it coefficient ideal} $C(I)$ and the auxiliary object $B_s''$ described in \cite{EV}, which require to take powers of ideals. Since our rings are not reduced, certain powers of elements that, by analogy with the classical case, should not be zero, sometimes vanish causing the mentioned difficulties. Thus, for application to deformations, it seems convenient to use a resolution algorithm  where these constructions 
are substituted by others that do not involve powers of ideals.

The purpose of this paper is to propose such a resolution algorithm.
 Unlike in \cite{EV} we 
 do not work with basic objects, but rather with   {\it multi-ideals}, which are, essentially, systems 
 $(W, (I_1,b_1), \ldots, (I_n,b_n),E)$, where for each $i$, $(W,I_i,b_i,E)$ is a basic object. 

Section \ref{S:multi}  introduces the basic notions of multi-ideals, transforms, equivalence, algorithmic resolution. In section \ref{S:tools} we study some useful tools, including  some ``differential calculus'' and the formalism to implement the inductive step mentioned above (``hypersurfaces of maximal contact''). Section \ref{S:A} discusses monomial ideals and the technique to reduce the general situation to one where induction can be applied. Section \ref{S:C}  presents our algorithm. The usual glueing problem in  the inductive step is handled by using the naturality, or functoriality,  properties of our algorithm (a technique started in \cite{BMF} and followed in \cite{T2}). Finally in  section  \ref{S:Ap},  we explain how our algorithm can be extended, to a certain extent, to the situation where we work 
over a suitable artinian ring rather than over a field. We also explain how  the  center permissibility problem is no longer present. 

 The idea to work with multi-ideals to avoid  this difficulty was suggested to the author by a referee of the article \cite{T1}. I thank this mathematician for the advice.

\section{Multi-ideals}

In this section we introduce the most basic notions that will be used in the article.

\label{S:multi}
\begin{voi}
\label{V:m1} 
Throughout we shall use the notation and terminology of \cite{H}, with a few exceptions that  we   explain next.

If $W$ is a scheme, a $W$-ideal means a coherent sheaf of ${\cO}_W$-ideals.  
If $Y$ is a closed subscheme of a scheme $W$, the symbol $I_W (Y)$ denotes the $W$-ideal defining $Y$.  If $W$ is a reduced scheme, a never-zero $W$-ideal is a $W$-ideal $I$ such that the stalk $I_x$ is not zero for all $x \in W$.  
 An algebraic variety over a field $k$ is a reduced algebraic $k$-scheme. We  work throughout with the class  $\cV$ of algebraic varieties defined over fields of characteristic zero but with minor changes we could work with the more general class of schemes $\cS$ introduced in \cite[8.1]{BEV}. A positive divisor in an algebraic variety $X$ is called a hypersurface of $X$. 

 The order of an ideal $I$ in a local ring $A$, with maximal ideal $M$,  is the largest integer $s$ such that $I\subseteq M^s$.  If $W$ is a noetherian  scheme, $I$ is a $W$-ideal, and $x \in W$, then $\nu _x(I)$  denotes the order of the ideal  $I_x$ of ${\cO}_{W,x}$. 

 Often we  consider functions $f$ from a set $S$ to a totally ordered set $\Lambda$. We let $\max \, (f)$ designate the maximum value of $f$ and $\ma \, (f)$ the set of points $x$ where $f(x)$ is the maximum. 

The natural, rational, complex numbers and the integers will be denoted by $\mathbb N$, $\mathbb Q$, $\mathbb C$ and  $\mathbb Z$ respectively.

\end{voi}

\begin{voi} 
\label{V:par}
Let $W$ be a regular variety. An {\it idealistic pair on W}, or a $W$-{\it pair}, or just a {\it pair}  is an ordered pair 
$\uP = (I,b)$ where $I$ is a $W$-ideal and $b$ a positive integer. 
 The {\it singular set} of $\uP$ is 
$  \{ z \in W: \nu _x(I) \ge b        \}   $. This is  a closed subset of $W$ (see, e.g., \cite{EV}),  denoted by 
$\sg (\uP)$ or $\sg(I,b)$. If $I$ is a never-zero $W$-ideal, then the open set $W \setminus \sg (I,b)$ is dense in $W$. 

A regular subscheme $C$ of $\sg (\uP)$ is called a {\it permissible center} for $\uP$. If $W_1$ is the blowing-up of $W$ with a permissible center $C$, then $W_1$ is regular and we may define  several $W_1$-ideals associated to $I$  called its {\it transforms}:
\begin{itemize}
\item[(i)] the {\it total transform} $I{\cO}_{W_1}$; 
\item[(ii)] the {\it controlled transform} $I[1]:={\cE}^{-b}{\cO}_{W_1}$ where $\cE$ defines the exceptional divisor of the blowing-up;
\item[(iii)] the  {\it proper (or weak) transform} ${\bbar I}[1]:=  {\cE}^{-a}{\cO}_{W_1}$ where the exponent $a$ is as large as possible.   If $C$ is irreducible with generic point $y$,  then $a$ is constant equal to $\nu _y(I)$ (in general $a$ is locally constant).
\end{itemize}
The pair (on $W_1$) ${\uP}_1 = (I[1],b)$ is called the permissible transform of $\uP$, sometimes denoted by $\uT (\uP,C)$.

A pair  $(I,b)$ is said to be {\it good} \cite{EV} or {\it simple} \cite{EVU}, if   $\nu_x(I)=b$ for all $x \in \sg(I,b)$.  In that case, for any permissible center, the controlled and proper transforms  of the pair coincide.

An ordered $n$-tuple $(I_1,b_1), \ldots, (I_n,b_n) $ of $W$-pairs will be called a {\it multi-pair}. By definition its singular set is the set 
$\bigcap _{i=1}^n \sg(I_i,b_i)$.
\end{voi}

\begin{voi}
\label{V:m2}
Let $M$ be a regular variety and $E = (H_1, \ldots , H_m)$ a sequence of regular hypersurfaces of $M$.

(a) $E$ have normal crossings if, for all $x \in H_1 \cup \cdots \cup H_m$, the ideal 
${I(H_1 \cup  \ldots \cup H_m)}_x \subset {\cO}_{M,x}$ 
is generated by $a_1 \ldots a _r$, for some $r$, $1 \le r \le n$, where 
$a_1, a_2, \ldots, a_n$
 is a suitable regular system of parameters  of $\cO _{M,x}$.

(b) We say that  a closed  subscheme $V  \subset M$ has normal crossings with respect to $E$ (resp. is transversal to $E$) if, for all $x \in V$, there is a regular system of parameters $a_1, \ldots, a_n$ of $\cO _{M,x}$, such that $I(V)_x=(a_1, \ldots, a_r){{\cO}_{M,x}}$, $1 \le r \le n$, and if for any hypersurface $H_j$ containing $x$, we have  $I(H_j) _x = (a_i){\cO}_{M,x}$ for some index $i$ (resp. for some index $i>r$). Such a subscheme $V$ is necessarily regular.
\end{voi}

\begin{Def}
\label{D:multi-ideal}  
 A {\it multiple marked ideal,} or simply, a {\it multi-ideal}, is a tuple 
$$\cI = (M,W,(I_1,b_1), \ldots, (I_n,b_n),E)$$
 where $M$ is a regular variety in $\cV$, 
$E=(H_1, \ldots, H_m)$ is an ordered $m$-tuple of distinct hypersurfaces of $M$ with normal crossings, 
$W$ is a   equidimensional subvariety of $M$ transversal to $E$, and $(I_j,b_j)$ is a  $W$-pair,  $j=1, \ldots, n$.  

 This multi-ideal is called {\it nonzero} if for all $x \in W$ there is an index $i$ such that the stalk ${(I_i)}_x$ is a nonzero ideal of ${\cO}_{W,x}$. 
 
 The scheme $W$, which is necessarily regular,  is called the {\it underlying scheme} of $\cI$ and is denoted by $us(\cI)$ while the scheme $M$ is called the {\it ambient scheme} of $\cI$ and is  denoted by $as(\cI)$.

The idealistic pairs  $(I_j,b_j)$ are called the {\it pairs} of $\cI$.
 By the assumed transversality, the restriction  
 $E|W:= (H_1 \cap W, \ldots, H_m \cap W)$ (where we ignore empty intersections) has again normal crossings on $W$. 

 If $n=1$ then  $\cI$ is called a {\it marked ideal}   
 or a {\it basic object}.

In the definition of multi-ideal, the most essential scheme is $W$, the underlying one, and we might avoid mentioning $M$, the ambient variety. This approach (followed, e.g., in 
\cite{EV} and \cite{T1}) would simplify the notation, and much of the theory could be developed  practically without changes. However, when in \ref{V:eqiv} we investigate the notion of {\it equivalence}, even  for  objects with different underlying schemes, it seems that the given definition is more convenient.
\end{Def}

\begin{Def}
\label{D:sing}
 The {\it singular set } of the multi-ideal $\cI = (M,W,(I_1,b_1), \ldots, (I_n,b_n),E)$ is 
 $\sg (\cI)=\{   x \in W : \nu _x(I_j ) \ge b_j, ~ j=1, \ldots n        \}$. We have   
 $\sg(\cI)= \bigcap _{j=1}^{n}   \sg ( I_j,b_j) $, which is  a closed set of $W$. If $\cI$  is nonzero,
 then, for any irreducible component $W'$ of $W=us(\cI)$,    we have    $\sg (\cI) \cap W' \not= W'$.

 If $\sg(\cI)$ is empty (resp. not empty), we say that $\cI$ is {\it resolved} (resp. {\it nonresolved}).
\end{Def}

\begin{voi}
\label{V:m3} We use the notation of \ref{D:multi-ideal}.

(a)  {\it Permissible centers and transformations}. 
  Given a multi-ideal $\cI$         we say that a closed subscheme $C$ of $W$ is a {\it permissible center} for  $\cI$, or that it is $\cI$-{\it permissible},  if $C$ has normal crossings with $E$ and  $C \subseteq \sg (\cI) $. An $\cI$-permissible center  $C$  is necessarily  a regular subscheme of $W$,  and $C$ is permissible for each of the pairs $(I_j,b_j)$.

If $C$ is $\cI$-permissible,  the transform of the multi-ideal $\cI$ with center $C$ is   the multi-ideal 
${\cI}[1] = (M_1, W_1, ({I_1}[1], b_1), \ldots, (I_[1],b_n,  E_1)$  
where  $M_1$ is the blowing-up of $M$ with center $C$, $W_1$ is the strict transform of $W_1$ (identifiable to the blowing-up of $W$ with center $C$), the $W_1$-pair $(I_{j}[1], b_j) $ is the  transform of the $W$-pair $(I_j,b_j)$, $j=1, \ldots, n$, and $E_1=(H'_1, \ldots,H'_m,H'_{m+1})$ (with $H'_i$ the strict transform of $H_i$, $i=1, \ldots, m$, and $H'_{m+1}$ the exceptional divisor).
 We write $\cI _1=\uT(\cI,C)$ and we denote a transformation of the multi-ideal $\cI$ by the symbol $\cI \leftarrow {\cI}_1$.

A sequence of multi-ideals and arrows
 $\cI _0 \leftarrow \cdots \leftarrow \cI_s $
is called a {\it permissible sequence} if each arrow stands for a transformation with a permissible center $C_j \subset us(\cI _j)$.

(b) {\it Pull-backs.} If $f: M' \to M$ is a smooth morphism,  the {\it pull-back} of the multi-ideal $\cI$ is the multi-ideal
 $f^*(\cI):=(M',W',(I_1{\cO}_{W'},b_1), \ldots,    (I_n{\cO}_{W'},b_n), E')$, 
 where $W'=f^{-1}(W)$ and $E'=(f^{-1}(H_1), \ldots , f^{-1}(H_m))$, ignoring the empty entries, if any. If $f$ is an isomorphism, we talk about an {\it isomorphism} of multi-ideals.  If $M'=U$ is an open subscheme of $M$ and $f$ is the inclusion, $f^*(\cI)$ will be called the restriction of $\cI$ to $U$ (usually denoted by ${\cI}|U$).

(c) {\it Extensions and open restrictions}.  Here are two important special cases of pull-backs. 
\begin{itemize}
\item[(i)]  $M'=M \times _k {\bf A}_k^1$ (where $M$ is defined over the field $k$) and $f$ the first projection. 
 The resulting  multi-ideal ${\cI}(e):=(M', W', I_1, b, E(e))$ 
   is called the {\it extension} of $\cI$ (see \cite{EV}).

 \item[(ii)] $W=U$ is an open subset of $M$ and $f$ is the inclusion. The resulting multi-ideal is the {\it restriction} of $\cI$ to $U$, denoted by 
${\cI}|U$. An {\it open restriction} of $\cI$ is a restriction to some open set of $M$.
\end{itemize}

(d) {\it Resolutions}. A {\it resolution} of a multi-ideal $\cI$ is a sequence of multi-ideals and permissible transformations, 
$\cI:= {\cI}_0 \leftarrow \cdots \leftarrow {\cI}_r$, 
such that $\sg ({\cI}_r) = \emptyset$. 
\end{voi}

\begin{voi} 
\label{V:eqiv}
{\it Equivalence}. A sequence 
$ \cI:= {\cI}_0 \leftarrow \cdots \leftarrow {\cI}_s$ is {\it local} if each arrow stands for a permissible transformation, or an open restriction, or an extension.

\smallskip

 Multi-ideals 
$\cI=(M,W,(I_1,b_1), \ldots, (I_n,b_n),E)$,     
$\cJ=(M,V,(J_1,b_1), \ldots, (J_p,b_p),E)$ 
are {\it equivalent}  (denoted $\cI \sim \cJ$) if:
\begin{itemize}
 \item [(a)] $\sg(\cI)=\sg(\cJ)$,
\item [(b)] Whenever  
$\cI:= {\cI}_0 \leftarrow \cdots \leftarrow {\cI}_s$ and 
$ \cJ:= {\cJ}_0 \leftarrow \cdots \leftarrow {\cJ}_s $
 are local sequences of multi-ideals, where corresponding arrows stand for the same type of operation (a permissible transformation with the same center, an extension, or a restriction to a common open set),   we have 
$\sg(\cI _s)=\sg(\cJ _s)$. Equivalently, instead of this equality of singular sets, we could require that a scheme $C$ be a permissible center for $   \cI _s   $ if and only if $C$ were a permissible center for $\cJ _s$.
\end{itemize}
\end{voi}

The above expression ``local sequence'' was used by O. Villamayor. The following result is easily proved.
\begin{pro}
 \label{P:abi}
If $\cI$ and $ \cJ$ are equivalent multi-ideals and  $U$ is an open set in $as(\cI)=as(\cJ)$, then 
$\cI |U \sim \cJ |U$.
\end{pro}

\begin{rem} 
 \label{R:abig}
Proposition \ref{P:abi} admits the following  simple generalization (see  \cite[1.10]{T2}).

\smallskip

{\it Let $\cI$ and $\cJ$ be equivalent multi-ideals. Suppose  that 
$\cI = {\cI}_0 \leftarrow  \cdots  \leftarrow {\cI}_s$ and 
$\cJ = {\cJ}_0 \leftarrow  \cdots  \leftarrow {\cJ}_s$  are permissible sequences of transformations, both obtained by using the same permissible centers (hence 
$M_i=as(\cI _i)=as(\cJ _i)$, for all $i$). Let  $U$ be any open set in $
M_s$. Then 
${\cI _s} _{|U} \sim {\cJ _s}_{|U} $.}

\smallskip

\end{rem}
\begin{voi}
\label{V:alre}
{\it Algorithmic resolutions}.  A resolution  algorithm for multi-ideals is the assignment, for each non-negative integer $d$, of a totally ordered set $\Lambda^{(d)}$, as well as the assignment, for each nonresolved multi-ideal $\cI$, of functions $g_0, \ldots, g_{r-1} $ (where $r$ depends on $\cI$). In the set $\{g_0, \ldots, g_{r-1}\}$,  
 $g_i$ is upper-semicontinuos with values in $\Lambda ^{(d)}$ (where $d$ is the dimension of $\cI$). The function $g_0$ is defined on $\sg (\cI)$ and 
  the closed set  
 $C_0:=\ma (g_0)$ 
 must be 
a permissible center for $\cI$. If $r > 1$, letting $\cI _1 := {\uT}(\cI _{0}, C_{0})$, 
 the domain  of
$g_1$ is $\sg (\cI _1)$,    and           $C_1:=\ma (g_1)$  must be a permissible center. Recursively, for     $j=0, \ldots, r-1$,   
 $g_j$ has    $\sg (\cI _j)$ as domain,   where $\cI _j = {\uT}(\cI _{j-1}, C_{j-1})$ and $C_{j-1}= \ma (g_{j-1})$ is a permissible center. The resulting permissible sequence 
$\cI \leftarrow \cI _1 \leftarrow \cdots \leftarrow \cI _r$ of multi-ideals must be a resolution, i.e., $\sg (\cI _r) =\emptyset$.

  We are interested in algorithms that also satisfy  the following compatibility conditions:  

(a) {\it Compatibility with open immersions}: if $\cI =(M,W,I,b,E)$ and $U$ is a open set in $M$, then the algorithmic resolution functions of $\cI$ induce those of ${\cI}|U$ (ignoring those arrows in the resolution of ${\cI}|U$ which are  isomorphisms). 

(b) {\it Compatibility  with equivalence}: if $\cI=(M,W,I,b,E)$ and $\cJ=(M,V,J,c,E) $ are  equivalent multi-ideals with $\dim (W)=\dim(V)$, then the algorithmic resolutions functions for $\cI$ and $\cJ$ are the same. 

From now on, a ``resolution algorithm'' will mean one satisfying conditions (a) and (b).

\end{voi}

\begin{rem}
\label{R:boj}
   Note that an algorithmic resolution for multi-ideals implies one for basic objects. Indeed, a permissible transform of a basic object  is again a basic object. 

\smallskip

Conversely, an algorithmic resolution  for basic objects implies one for multi-ideals, a  result which is a consequence of the  following key observation, suggested by O. Villamayor:   
 given a multi-ideal   $\cI=(M,W,(I_1,b_1), \ldots, (I_n,b_n),E)$, there is  a marked ideal or basic object $B_{\cI}$ that is equivalent to $\cI$.
 
This basic object,  called the {\it basic object associated to the multi-ideal $\cI$}, is constructed as follows. 
 Let $N$ be the least common multiple of $b_1, \ldots, b_n$ and  write $N=q_i b_i$, $i=1, \ldots, n$, for suitable integers $q_1, \ldots, q_n$, $J=I_1^{q_1}+ \cdots + I_n ^{q_n}$. Then   $B_{\cI}=(M,W,(J,N),E)$. The pair $(J,N)$ is the sum of the pairs $(I_1,b_1), \ldots, (I_n,b_n)$, in the terminology of \cite[3.3]{BMF};  
in \cite{EVU} this operation of sum is denoted by the symbol  $\odot$.

The verification of the fact that $\cI \sim B_{\cI}$ is straightforward.
.\end{rem}

\begin{voi}
\label{V:ree}
 Multi-pairs seem to be related to Rees algebras. A Rees algebra 
$\cG$ 
over a scheme $W$ is determined by a family of $W$-ideals $(I_i)$, $i=0, 1,  \ldots$, $I_0={\cO}_W$. Namely, 
$\cG = \oplus _i I_i T^i$, a graded subring of ${\cO}_W[T]$ (where $T$ is an indeterminate), it is required that $\cG$ be locally finitely generated (see \cite{EVU}). 
 So, a multi-pair looks like an ``embryonic form'' of a Rees algebra. However, although Rees algebras have a rich structure, very useful in the theory of resolution of singularities working over fields, they do not seem to behave so well when we work over a base ring which is not reduced. For instance, if $A=k[\epsilon]$, with $k$ a field and $\epsilon ^2 =0$, and we consider the inclusion $A \subset A[T]$ ($T$ an indeterminate), then the element 
${\alpha}_n := \epsilon T^n$ is integral over $A$, for all $n$. Indeed, 
$\alpha _n$ is a root of the monic polynomial $X^2 \in A[T]$. From this observation it is easy to produce examples of Rees algebras 
$\cG \subset {\cO}_W[T]$ whose integral closure is not a finite module. Over fields, the integral closure plays an important role in the theory of equivalence of Rees algebras. Often operations with Rees algebras involve taking powers, which, as explained in the introduction, sometimes create difficulties when working with rings with nilpotents. 

We believe that multi-ideals (or a variant thereof) might be a more suitable tool for applications in situations involving non-reduced rings.                
\end{voi}

\section{Some useful tools} 
\label{S:tools}

In this section we discuss some preliminary concepts that we need to construct  an algorithm for resolution of multi-ideals.

\begin{voi}
\label{V:del} 
 {\it The $\Delta$ operation}. Given a coherent sheaf of ideals $I$ over a regular variety $W$ (over a zero characteristic field $k$), it is possible to introduce auxiliary sheaves 
${\Delta}^{(j)}(I)$, $j=0, 1, \ldots$, which play an important role in the theory of resolution of singularities. 
 For any integer $j \ge 0$, 
 ${\Delta}^{(j)}(I)$ has the following  property: if $w$ is a   closed point of 
$ W$ and we choose any regular system of parameters $x_1, \ldots, x_d$ of ${\cO}_{W,w}$, then the stalk
${\Delta}^{(j)}(I)_w$ is the ideal of ${\cO}_{W,w}$ generated by all the elements $f \in I_x$ as well as by their partial derivatives, up to order $j$,  with respect to $x_1,\ldots,x_d$. In particular,   ${{\Delta}^{(0)}(I)}=I$. These sheaves may be globally defined with the aid of suitable 
 Fitting ideals (\ref{V:ap5}, (a)).    Another construction can be found in \cite[6.1]{Cu}. If $(I,b)$ is a $W$-pair, then $\sg(I,b)=V({\Delta}^{(b-1)}(I))$. 
\end{voi}

\begin{voi}
\label{V:ut1} {\it Coefficient multi-ideal}.   Let 
$ \cI =( M,W,(I_1,b_1), \ldots, (I_n,b_n), E        )$ be a nonzero multi-ideal. The multi-ideal
$$(1) \quad  {\mathcal C}(\cI) =( M,W,(I_1,b_1),     ({\Delta}^{(1)}(I_1),b_1-1)    ,  \ldots, ({\Delta}^{(b_1-1)}(I_1),1),    \ldots, $$
$$(I_n,b_n),   ( {\Delta}^{(1)}(I_n),b_n-1)    ,  \ldots,({\Delta}^{(b_n-1)}(I_n) ,1),       E        )$$

\noindent
is called the {\it coefficient multi-ideal} associated to $\cI$. That is, in $ {\mathcal C}(\cI)  $ each pair $(I_i,b_i)$ in  $\cI$ is substituted by the multi-pair 
$(I_i,b_i),   ( {\Delta}^{(1)}(I_i),b_i-1)    ,  \ldots,({\Delta}^{(b_i-1)}(I_i) ,1)$. 

\smallskip

It is easily verified that 
$\sg (J,c)= \sg (\Delta ^{(c-q)}(J), q)$, $q=1, \ldots , c$, a  result that implies the equality  
$\sg(\cI)=\sg(\cC(\cI))$. 

\smallskip

If $Z$ is a  hypersurface of $W$   transversal to $E$, then 
$$(2) \quad  {{\mathcal C}(\cI)}|Z =( M,Z,(I_1|Z,b_1),     ({\Delta}^{(1)}(I_1)|Z,b_1-1)    ,  \ldots, ({\Delta}^{(b_1-1)}(I_1)|Z,1)   \ldots, $$
$$(I_n|Z,b_n),   ( {\Delta}^{(1)}(I_n)|Z,b_n-1)    ,  \ldots,({\Delta}^{(b_n-1)}(I_n)|Z ,1),       E|Z        )$$

\noindent is a  multi-ideal of dimension $d-1$, if $\cI$ had dimension $d$. If ${{\mathcal C}(\cI)}|Z $ is nonzero, we call it the {\it inductive multi-ideal} induced by $\cI$ on $Z$, also  denoted by the symbol
$\cI _Z$.  
In this case we say: {\it the inductive multi-ideal is defined}.

\smallskip

 Given  $\cI$ as above,  the following condition on a hypersurface $Z$ of $W$ insures that 
the inductive multi-ideal ${\cI}_Z$ be defined:

\begin{itemize}
\item[($\iota$)] For each $x \in Z$, there is an index $i$ such that 
${({\Delta}^{(b_i-1)}(I_i)|Z)}_x \not= 0$ and $Z$ is   transversal to $E$.
 \end{itemize}
\end{voi}

\begin{pro}
\label{P:sin}
Let $ \cI =( M,W,(I_1,b_1), \ldots, (I_n,b_n), E )$ be a 
 multi-ideal and $Z$ a regular hypersurface  of $W$ satisfying condition ($\iota$).
   Then, 
$\sg(\cC(\cI) |Z) = \sg(\cI) \cap Z \, .$
\end{pro}

\begin{proof} 
 {\it (a) The inclusion $\sg(\cC(\cI) |Z) \subseteq \sg(\cI) \cap Z$}. 
  Let $y \in \sg(\cC(\cI) |Z)$. Then certainly $y \in Z$. Let us show that 
$y \in  \sg(\cI) $. We may assume
 that $y$ is a closed point of $Z$. We have to prove that $\nu _y (I_j) \ge b_j,~ j=1, \ldots, n$. To this end, let $x_1, \ldots, x_d=z$ be a regular system of parameters of ${\cO}_{W,y}$, where $Z$ is defined by $z$ at $y$. We may work on the completion $R$ of ${\cO}_{W,y}$, which is a power series ring 
 $k'[[x_1,\ldots, x_{d-1},z]]$, where  $k'$ is a suitable finite extension of the base field $k$.  
Then the completion $R'$ of ${\cO}_{Z,y}$ is isomorphic to $k'[[x_1, \ldots, x_{n-1}]]$.
 It suffices to show that if 
 $g \in I_j R$, then $\nu _y (g) \ge b_j, ~ j=1, \ldots, n$. Write $g \in I_jR$ as a power series 
 $g=a_0 +a_1z + \cdots + a_{b_j -1}z^{b_j -1}+ \cdots$. Making $z=0$, we obtain $a_0$, which is an element of $(I_j|Z)R'$. So, by assumption,  
 $\nu _y(a_0 ) \ge b_j$. Similarly, taking derivative with respect to $z$ of order $i < b_j$, making $z=0$, and since we work in characteristic  zero, we see that  $a_i \in  
 (  {\Delta ^{b_j -i} I} _j |Z)R'$. Hence if  $i < b_j$, then $\nu_y(a_i) \ge b_j -i$ and  thus $\nu _y (g) \ge b_j$, as needed.

 {\it (b) The inclusion
 $ \sg(\cI) \cap Z     \subseteq     \sg(\cC(\cI) |Z) $}. 
  Let $y \in \sg(\cI) \cap Z$.  We may assume $y$ is a closed point. Using the notation of part $(a)$, since $y \in Z$, we may assume that we have a regular system of parameters $x_1, \ldots, x_d=z$, where $z$ defines $Z$ at $y$. It suffices to show that  working in the completions $R$ of ${\cO}_{W,y}$ and $R'$ of ${\cO}_{Z,y}$, if ${g'} \in ({\Delta}^{i}(I_j)|Z)R'$, then 
 $\nu _y ( g') \ge  b_j - i    $. Moreover, we may assume that $g'=D(h)_{|z=0}$, where 
 $h \in I_j R$ and 
$D= {{\partial}^s}/{\partial {x_{j_1}}\ldots  \partial {x_{j_s}}}$, $s \le i$. 
 Expressing $h$ as a power series $h=a_0 +a_1z + \cdots + a_{b_j -1}z^{b_j -1}+ \ldots$ and   taking the indicated partial  derivatives, it is easy to show, as in part (a),  that   $D(h)= g' + g'_1z + \cdots$, where the order of  $g' _j$ in $R'$ is $\ge b_j -i$, as claimed.
\end{proof}

\begin{cor}
\label{C:equaz}
Let $ \cI =( M,W,(I_1,b_1), \ldots, (I_n,b_n), E )$ be a 
 multi-ideal and   $Z$ a regular hypersurface  of $W$ satisfying Condition ($\iota$) and containing $\sg(\cI)$. Then 
$\sg(\cC(\cI) |Z) = \sg(\cI) $
\end{cor}

This corollary follows from \ref{P:sin}. The equality of \ref{C:equaz} remains valid after taking successively permissible transformations. This generalization is Proposition \ref{P:alfin}, but to state it precisely and prove it we need  some preliminary material to be discussed next. We begin by recalling an important result about the operators $\Delta ^{(i)}$. 

\begin{pro}
\label{P:gira}
Let $\uP=(J,c)$ be   a $W$-pair ($W$ a regular variety) and   $C$ a $\uP$-permissible center.  Consider the transform 
${\uP}_1=(J[1],c)$ of $\uP$ with center $C$ and let 
$\cE \subset {\cO}_{W_1}$ denote the $W_1$-ideal defining the exceptional divisor. Then  
$${\cE}^{-i} {\Delta}^{(c-i)}(J){{\cO}_{W_1}} \subseteq {\cE}^{(i)} {\mathrm {and}} 
 {\cE}^{-i} \Delta^{(c-i)}(J){\cO}_{W_1} \subseteq   \Delta ^{(c-i)}(J[1]) \, .$$
\end{pro}

The above proposition (see a proof in \cite[6.6]{EV})  compares the pairs 
$(\Delta ^{(c-i)}(J),i)$ and 
$(\Delta ^{(c-i)}(J[1]),i)$. Namely,  it says that the ideal of the transform of the first pair (with a permissible center $C$) is contained in the ideal of the second pair.

\smallskip

We also need a basic result on restrictions of pairs whose simple proof we omit.

\begin{pro}
\label{P:restr} 
Let $W$ be a regular variety (in $\cV$),  $(J,c)$ be a $W$-pair, $(J[1],c)$ its transform with a permissible center $C$ (which is a $W_1$-pair, where  $W_1$ is the blowing-up of $W$ with center $C$). Let 
 $Z$ be a regular hypersurface in $W$ containing $\sg (J,c)$ and  $Z_1$ the strict transform of $Z$ to $W_1$ (which can be identified to the blowing-up of $Z$ with center $C$). Let the $Z_1$-pair 
$((J|Z)[1],c)$ be the transform of the $Z$-pair 
 $(J|Z,c)$ with center $C$  
 (with $J|Z$ the restriction of $J$ to $Z$). Then, 
$(J|Z)[1]  =     (J[1]|Z_1)$
\end{pro}

 \begin{voi}
\label{V:asu} Here we describe a situation that will appear later on in the paper. 

Let $\cI=(M,W,(I_1,b_1), \ldots, (I_n,b_n),E) $ be a multi-ideal and consider 
 a permissible sequence of multi-ideals
$$(1) \qquad  {\cI} = {\cI}_0   \leftarrow \cdots \leftarrow \cI _{s} ~,  s>1 \, .    $$ 
We write 
 $ \cI _j = (M_j, W_j, (I_1[j],b_1),\ldots,(I_n[j],b_n), E[j])$, $0 \le j \le s$ and we let $C_j$ denote the $j$-th permissible center used in (1).
  We assume that we have  hypersurfaces 
$Z_j$ of $W_j$ such that      $\sg(\cI _j) \subseteq Z_j $,    $j=0, \ldots, s$ and that, for all $j$,   $Z_j$  is the strict transform of $Z_{j-1}$ (via the morphism $W_{j-1} \leftarrow W_j$ determined by (1)). Moreover we suppose that  $Z:=Z_0$ satisfies  condition ($\iota$) of \ref{V:ut1} and   hence, the inductive multi-ideal 
$\cI _Z$ is defined. We also assume that the sequence (1) induces a permissible sequence 
$$(2) \qquad   {\cI}_Z :=     [{\cI}_Z]_0 \leftarrow \cdots  [{\cI}_Z]_s ~,   $$ 
by using the same centers $C_j$ that we used in (1), in such a way that $\sg(\cI _j)=\sg([{\cI}_Z]_j)$, 
$j=0\ldots, s-1$.

Let us better explain how (1) induces (2).  We  suppose that  
$ \sg(\cI _0)=\sg([{\cI}_Z]_0)  $, so that $C_0$  is a permissible center for both $\cI_0$ and ${[\cI _Z]}_0$; if $\cI _1 = \cT(\cI,C_0)$ and 
 ${[\cI _Z]}_1 = \cT ({[\cI _Z]}_0,C_0)$, then we have $\sg(  {\cI}_1 ) = \sg  ( {[\cI _Z]}_1  )  $, and so on.

 Note that the assumptions above imply that $Z_j$ satisfies condition ($\iota$), $j=0, \ldots, s$.
\end{voi}

\begin{voi}
\label{V:ut2} 
Let us study the pairs of the multi-ideal $[{\cI}_Z]_j$ appearing in the sequence \ref{V:asu}(2).

First, using the notation of \ref{V:asu}, consider the multi-ideal $\cC(\cI)={\cC(\cI)}_0$. Since $\cI$ and $\cC(\cI)$ are equivalent (see, e.g., \cite[3.11]{BMF}), the sequence \ref{V:asu} (1) induces, using the same centers for the transformations, a sequence 
 $ \cC(\cI)={\cC(\cI)}_0 \leftarrow \cdots \leftarrow             {\cC(\cI)}_s$.

Let us denote the pairs of ${\cC(\cI)}_j $ by 
$({[\Delta ^{(b_i-q)}(I_i)]}_j , q)$, $0 \le q < b_i, ~ 1 \le i \le n$. The ideals are inductively constructed as follows:  
$$  {[\Delta ^{(b_i-q)}(I_i)]}_0 ={\Delta ^{(b_i-q)}(I_i)}  ~   {\mathrm   {and} } ~
{[\Delta ^{(b_i-q)}(I_i)]}_{j+1} := {{\cE}_{(j+1)}}^{-q}{[\Delta ^{(b_i-q)}(I_i)]}_j \, ,$$ 
where $\cE _{j+1}$ denotes the $W_{j+1}$-ideal defining the exceptional divisor of the morphism $W_j \leftarrow W_{j+1}$ determined by the sequence \ref{V:asu} (1). Then, by repeatedly 
applying  Proposition \ref{P:restr}, we see that the pairs of $[{\cI}_Z]_j$ are the $Z_j$-pairs 
$$  ({[\Delta ^{(b_i-q)}(I_i)]}_j |Z_j , q)    \,  ,~ 0 \le q < b_i, ~ 1 \le i \le n ~.  $$ 
To simplify, we shall write 
 ${[\Delta ^{(b_i-q)}(I_i)]}_{Z,j} := {[\Delta ^{(b_i-q)}(I_i)]}_j |Z_j $.
 Notice that, by \ref{P:restr}, the pair 
 $({[{\Delta}^{(b_i-q)}(I_i)]}_{Z,j+1},q)$ appearing in ${[\cI]}_{j+1}$, is the controlled transform (\ref{V:par}) of the pair $({[{\Delta}^{(b_i-q)}(I_i)]}_{Z,j},q)$.

\end{voi}

\begin{pro}
\label{P:aninclu}
With the notation of \ref{V:asu} and \ref{V:ut2}, we have:

(a)  an inclusion of $W_j$-ideals 
$${ [\Delta ^{b_i-q}(I_i)]}_j     \subseteq     \Delta ^{b_i-q}(I_i[j]) , ~ 0 \le q < b_i, ~ 1 \le i \le n, ~j=0, \ldots, s\, ,$$

(b) an inclusion of $Z_j$-ideals 
$${ [\Delta ^{b_i-q}(I_i)]}_{Z,j }    \subseteq     \Delta ^{b_i-q}(I_i[j])|Z_j , ~ 0 \le q < b_i, ~ 1 \le i \le n, ~j=0, \ldots, s.$$
\end{pro}

\begin{proof}
 Part (a) is seen by induction on $j$. The inclusion is clearly true if $j=0$. Assuming that (a) is 
valid for the index $j$, let $\cE $ denote the ${\cO}_{W_{j+1}}$-ideal defining  the exceptional divisor of the  
 blowing-up morphism $W_j \leftarrow W_{j+1}$, induced by the sequence 2.7 (1). Then we have 
 $${[{\Delta}^{b_i-q}(I_i)]}_{j+1} = {\cE}^{-q}   { [\Delta ^{b_i-q}(I_i)]}_j \subseteq     {\cE}^{-q}   ({ \Delta ^{b_i-q}(I_i [j])}) \subseteq  
           { \Delta ^{b_i-q}(I_i [j+1])} ,         $$
where the first inclusion is valid by the inductive hypothesis and the second one by \ref{P:gira}. So, (a) is proved for the index $j+1$.

Part (b) follows from (a) by restricting to $Z_j$.
\end{proof}

\begin{pro}
\label{P:despe}
With the assumptions and notation of \ref{V:asu} and \ref{V:ut2}, let $y$ be a closed point of $\sg(\cI _s)$. We write  $R=\widehat{{\cO}_{{W_s},y}}$ and 
$ {\bar R}=\widehat{{\cO}_{{Z_s},y}}$. Then, there is a regular system of parameters $x_1, \ldots, x_{d-1}, x_d$ of ${\cO}_{W_s,y}$ with the following properties: 

$(\alpha)$  Near $y$, $Z_s$ is defined by $z:=x_d=0$ (i.e.,  $I_{W_s}(Z_s)_y = (z){\cO}_{W_s,y}$) 

$(\beta)$ For $i=1, \ldots, n$, there are generators  $f^{(i)}_1,  \ldots, f^{(i)}_{m_i}$ of the ideal $(I_i[s])_y \subset {\cO}_{W_s,y}$ such that, as a power series in the completion $R=k[[x_1,\ldots,x_{d-1},z]] of {\cO}_{W_s,y}$,  we have $f^{(i)}_q = a_0 + a_1 z + \cdots + a_{{b_i}-1} z^{{b_i}-1} + \cdots $, $q=1, \ldots, m$, with each coefficient $a_j \in {\bar R}=k[[x_1, \ldots, x_{d-1}]]$ and, moreover, $a_j \in 
{(\Delta ^{(j)}(I_i[s])|Z_s)}{\bar R}$, 
 $0 \le j < b_i $.  (Each coefficient $a_j$ depends on $q$,  $k$  is a suitable field).
\end{pro}

The proof, by induction on $s$, is a simple variation of the demonstration of Theorem 6.24, Part 2, of \cite{Cu}.

\smallskip

Next,  using the notation of \ref{V:asu}, we  present the mentioned generalization of Corollary \ref{C:equaz}. 

\smallskip

\begin{pro}
\label{P:alfin}
Assume we are in the situation of \ref{V:asu}. Then,
$$(1) \qquad  \sg(\cI _s)= \sg ({[{\cI}_Z]}_s) \, .  $$
\end{pro}
\begin{proof}
The inclusion ${[{\Delta}^{b_i-q}(I_i)]}_{Z,s} \subseteq {{\Delta}^{b_i-q}(I_i[s])}|Z$ (\ref{P:aninclu}) implies an inclusion of singular loci of pairs:
$$ \sg({{\Delta}^{b_i-q}(I_i[s])}|Z,q)\subseteq \sg (   {[{\Delta}^{b_i-q}(I_i)]}_{Z,s} ,q),$$
which, by \ref{V:ut2}, implies:
$$  \sg(\cC (\cI _s) |Z  ) \subseteq \sg ({[\cI _Z]}_s)    \, .  $$
By \ref{C:equaz} and \ref{V:ut1}, $\sg ({\cC(\cI)}|Z  )   = \sg( \cC (\cI _s )) = \sg ({\cI} _s    )$. Hence, we obtain an inclusion 
$$    \sg({\cI}_s) \subseteq \sg(  {[\cI _Z]}_s        ).      $$
To prove the opposite inclusion, it suffices to work with closed points. So, let $y \in \sg(  {[\cI _Z]}_s )$ be a close point. We want to see that 
$\nu _y ({ I_i}[s]   ) \ge b_i $, $i=1, \ldots, n$. It suffices to show that $\nu _y ( f^{(i)}_j   ) \ge b_i$  for each of the generators of \ref{P:despe}. Now, working in the completion $R=k[[x_1, \ldots, x_d]]$ of $\cO_{W_s,y}$ (see \ref{P:despe}) we have:  $      f^{(i)}_j  =\Sigma \, a_t z^t$ (with $z=x_d$),   where 
 $a_t \in {\Delta ^{(t)}(I_i)} k[[x_1,\ldots,x_{d-1}]]$,
 $t=0, \ldots, b_i - 1$. Since $y \in \sg({[\cI]}_s)$,  $\nu _y (a_t) \ge b_i -t$. Thus each term of the series has order $\ge b_i$, and so $\nu _y (f^{(i)}_j) \ge b_i$, as needed.
\end{proof}

Next we shall study conditions under which the situation of \ref{V:asu} is reached. 

\begin{voi}
\label{V:ut3}
A basic object $B=(M,W,(I,b),E)$ is said to be {\it nice} if it is nonzero and there is a regular hypersurface
$Z$ of $W$ such that:

(a) ${I_W}(Z) \subseteq {\Delta ^{(b-1)}}(I)$

(b) $Z$ is transversal to $E$ (see \ref{V:m2} (a)).

Such a hypersurface   $Z$ is called, following \cite{T1},  an {\it adapted hypersurface} for $B$ or an $B$-{\it adapted hypersurface}. 

If $Z$ is $B$-adapted,  then at each point $y \in Z$ the stalk ${I(Z)}_y$ is generated by an order one element of ${{\Delta ^{b-1}}(I)}_y$.   Moreover, 
${\nu _y}(I)= b$ for all $y$ in $\sg(I,b)$. An $W$-pair with this property is called {\it of maximal order} or {\it good} or {\it  simple} (see \ref{V:par}). 

A  nonzero basic objectal $B$ is {\it locally nice} if for all $x \in W$ there is an open neighborhood
$U$ of $x$ (in $M$) such that the restriction ${B}|U$ is nice.

A nonzero multi-ideal 
$\cI =( M,W,(I_1,b_1), \ldots, (I_n,b_n), E )$ is said to be  $i$-{\it nice} (resp. $i$-{\it locally nice}),  
$i \in \{ 1, \ldots,n\}$, 
 if the associated marked ideal 
$( M,W,(I_i,b_i), E )$ is nice (resp. locally nice). A hypersurface $Z$ is {\it $i$-adapted} for $\cI$ if it is adapted for $( M,W,(I_i,b_i), E )$, and $Z$  is {\it adapted} if is $i$-adapted for some index $i$. The multi-ideal $\cI$ is nice if there is an index $i$ such that $\cI$ is $i$-nice; it is locally nice if each point of $M$ has a neighborhood $U$ such that the restriction of $\cI$ to $U$ is nice.

If $\cI$ is a  nice multi-ideal of dimension $d$ and the dimension of $\sg(\cI)$ is $< d-1$, then any adapted hypersurface necessarily satisfies condition $(\iota)$ of \ref{V:ut1}. Hence the inductive multi-ideal ${\cI}_Z$ is defined. 
\end{voi}

\begin{rem}
\label{R:getasu}
 We assume that  $\cI$ (as in \ref{V:ut3}) is a $v$-nice  multi-ideal, of dimension $d$, where 
 $Z$ is a $v$-adapted hypersurface, for some index $v \in \{  1, \ldots, n \}$. We also assume that  
$$(1) \qquad \dim \sg (\cI) < d-1 ~.$$
Then any permissible sequence  \ref{V:asu} (1)  leads to a situation where  the conditions described in  \ref{V:asu} are satisfied. This is obtained by taking, in \ref{V:asu}, $Z_0=Z$ and  $Z_j$ the strict transform of $Z_{j-1}$,  for all $j$. 

In other words, take a permissible transformation
$$\cI \leftarrow {\cI}_1 =    ( M_1,W_1,(I_1[1],b_1), \ldots, (I_n[1],b_n), E ) \, ,        $$ 
with center $C=C_0 \subseteq \sg(\cI)$  and  consider    the strict transform  $Z_1$  of $Z$ to $W_1$.   By our assumption (1), the inductive multi-ideal $\cI _Z$ is defined and, by Corollary \ref{C:equaz}, $C$ is an ${\cI}_Z$-center. Let ${\cI}_Z := {[\cI]_Z}_0 \leftarrow  {[\cI _Z]}_1$ be the corresponding permissible transformation.    Notice that  $Z_1$ is again $v$-adapted for ${\cI}_1$. Indeed, $C$ is permissible  for the 
$W$-pair $(I_v,b_v)$, because $\sg(\cI) \subseteq \sg(I_j,b_j)$ for all $j$. Since $Z$ is $v$-adapted and
$ {I_W}(Z) \subseteq {\Delta}^{b_v-1}(I_v)$, we have  
$${I_{W_1}}(Z_1)= {\cE}^{-1} {I_W(Z)}{\cO}_{W_1}\subseteq {\cE}^{-1}{\Delta}^{b_v-1}(I_v) \subseteq {\Delta}^{b_v-1}(I_v[1]),$$
where $\cE$ is the ideal defining the exceptional divisor and the last inclusion is obtained by using \ref{P:gira}. Thus (a) of the definition of $v$-adapted hypersurface is fulfilled. Part (b) of that definition follows from the fact that $C$ has normal crossings with  $E$. Now, by using Corollary \ref{C:equaz} and Proposition \ref{P:alfin} (for $s=1$), we see that 
$\sg ({\cI}_1) = \sg ({[\cI _Z]}_1)$.

Similarly, repeated applications of \ref{C:equaz}, \ref{P:gira} and \ref{P:alfin} show that  the requirements of \ref{V:asu} are satisfied 
if we take the hypersurfaces  $Z_j$ as indicated above. 
\end{rem}
 
\begin{pro}
 \label{P:heq} Consider a permissible sequence of multi-ideals 
$\cI _0 \leftarrow \cdots \leftarrow \cI _s$, where $\cI _s$ is nice, with adapted hypersurface $Z$. Assume  $\sg(\cI)$ has codimension $>1$ and let 
${(\cI _s)}_Z$ be the corresponding inductive (nonzero) multi-ideal. Then $\cI _s$ and  ${(\cI _s)}_Z$  are equivalent.
\end{pro}

\begin{proof}
This is  a consequence of the discussion in \ref{P:alfin} - \ref{R:getasu}. See  \cite[3.6]{T2} for details. 
\end{proof}

\smallskip

We have made several times the  assumption  that     
$\dim \sg (\cI) < \dim (\cI) - 1 $. But this hypothesis is not really restrictive  for resolution purposes, as we will see in the proposition below.

\begin{pro}
\label{P:cuno}
Let $\cI$ be a locally nice multi-ideal of dimension $d$ and assume that $C \not= \emptyset$ is the union of the irreducible components of $\sg (\cI)$ of dimension $d-1$. Then, $C$ is a permissible $\cI$-center. Moreover, if $\cI _1=\cT(\cI,C)$, then either $\sg (\cI _1)=\emptyset$ or $\sg(\cI _1)$ has no irreducible components of dimension $d-1$.
\end{pro}
The proof is almost identical to that of  \cite[3.3]{T2},
 where the case of marked ideals is considered.

\begin{voi}
\label{V:ut4}
  Assume $\cI$ is a nice multi-ideal of dimension $d$, such that $\dim  \sg (\cI) \le d-2$, with an adapted hypersurface $Z$. We shall see that, under the hypothesis that we have an algorithmic resolution process valid for multi-ideals of dimension $< d$, our previous results imply that this process extends to $\cI$. 

First note that,  as mentioned in \ref{V:ut3}, 
 $\cI$ automatically satisfies condition ($\iota$) of \ref{V:ut1} and thus the inductive multi-ideal ${\cI}_Z$ is defined. 
Since  multi-ideals of dimension is $< d$ are supposed to be algorithmically solvable,  $\cI _Z$ has an algorithmic resolution 
$$   (1) \quad  \cI _Z = {[{\cI}_Z]}_0 \leftarrow \cdots \leftarrow {[{\cI}_Z]}_r \, ,$$
obtained using algorithmic resolution functions ${\bar h}_0, \ldots, {\bar h}_{r-1}$ (with values in an ordered set $\Lambda^{(d-1)}$) and centers 
$C_j=\ma ({\bar h}_j)$, $j=0, \ldots, r-1$. 
Then, by Proposition \ref{P:alfin}, the sequence (1) induces a permissible sequence 
$$(2) \quad \cI = {\cI}_0 \leftarrow \cdots \leftarrow {\cI}_r \, ,$$
using the same centers, so that  
$\sg({{[\cI _Z]}_j})= \sg ({\cI}_j)$ for all $j$. Thus 
$\sg({{[\cI _Z]}_r})=\emptyset$ implies $\sg(\cI _r)=\emptyset$,  and so (2) is a resolution of $\cI$.  Setting $h_j={\bar h}_j$, 
$j=1, \ldots, r-1$, we get functions from 
 $\sg({{[\cI _Z]}_j})= \sg ({\cI}_j)$ to  $\Lambda^{(d-1)}$,  which are  algorithmic resolution functions for $\cI$, giving  the resolution (2).
\end{voi}

\begin{voi}
\label{V:ut5}
Now we consider a locally nice multi-ideal $\cI$ such that $\dim (\sg(\cI)) \le d-2$. We still  assume that an algorithm of resolution satisfying conditions (a) and (b) of \ref{V:alre} is available when the dimension of the marked ideal is $<d$.
 We shall prove that  $\cI$ can  be inductively resolved. 

There is an open cover $U_v$ of $M=as(\cI)$ (\ref{D:multi-ideal}) such that ${\cI}|U_v$ is nice, with an adapted hypersurface $Z_v \subset W \cap {U_v}$ 
for all  $v$. By induction on the dimension, 
 we have resolution functions 
$h^{(v)}_j$ for ${\cI}|U_v$, for all $v$, as in \ref{V:ut4}.
 We claim that the different functions $h^{(v)}_j$ agree on intersections, determining globally defined resolution functions (with values in $\Lambda ^{(d)}$) for the multi-ideal $\cI$. 

Indeed, consider an intersection 
$U=U_v \cap U_{w}$, let $Z'_v = Z_v \cap U_v$ and $Z'_w= Z \cap U_w$. We have: 
${\cI}_{Z'_v} \sim {\cI}|U \sim {\cI}_{Z'_w}$ (see Proposition \ref{P:heq}). So 
${\cI}_{Z'_v} \sim {\cI}_{Z'_w}$ and hence, by Condition (b) of our algorithm,  their algorithmic resolution functions agree. Since by Condition (a) the resolution functions for 
${\cI}_{Z'_v}$  and   ${\cI}_{Z'_w}$            are respectively the restrictions of those for ${\cI}_{Z_v}$     and ${\cI}_{Z_w}$, we conclude that $h^{(v)}(x)=h^{(w)}(x)$ for all $x \in U_v \cap U_w$.

Furthermore, we claim that the algorithmic resolution for locally nice multi-ideals just defined  
 also satisfies conditions (a) and (b) of \ref{V:alre}. The verification of (a) is straightforward. To verify (b), given  two equivalent  locally nice multi-ideals
$$\cI = (M,W,(I_1,b_1), \ldots, (I_n,b_n),E),~  
 \cJ={(M,V,(J_1,c_1), \ldots, (J_q,c_q),E)}\, , $$
both of the same dimension,  
   by (a), it suffices to show the following statement.  If $U \subseteq M$ is an open set such that 
$\cI |U$  and $\cJ |U$ are nice,  with adapted hypersurfaces $Z \subset W \cap U$ and 
$Z' \subset V \cap U$ respectively, then the resolution functions of 
$\cI |U$ and $\cJ |U$ agree. 

But this assertion is true because, by \ref{P:abi},   
$\cI \sim \cJ$ implies ${\cI}|U \sim {\cJ}|U$ and, by \ref{P:heq} and transitivity, 
${(\cI |U)}_Z  \sim   {(\cJ |U)}_{Z'}$. Since the dimension dropped by one, by induction, the resolution functions of 
${(\cI |U)}_Z$ and   ${(\cI |U)}_{Z'}$ agree. Hence,  by their inductive definition, those of 
$\cI |U$ and $\cJ |U$ also are the same.
\end{voi}

\section{Auxiliary objects}
\label{S:A}
In this section we study some objects useful in the construction of our algorithm. First, we discuss {\it monomial multi-ideals}, a class for which resolution is easily achieved. Later, we explain some auxiliary constructions that will allow us to reduce the problem of resolving general ideals to the monomial situation. 

\begin{voi}
 \label{V:am1}
{\it Monomial multi-ideals}  (\cite[5]{EV} {\it and}  \cite[6.16]{Cu}).  Consider a multi-ideal
$$\cI =(M,W,(I_1,b_1), \ldots, (I_n,b_n),E=(H_1, \ldots, H_m))\, .$$
 We say that $\cI$ is monomial if, for each $x \in \sg(\cI)$, there is an index $i \in \{1,\ldots,n \}$ depending on $x$,  such that  
$$(\star)\qquad {(I_i)}_x=      {{I_W(H_1 \cap W)}_x}^{\alpha_{i1}(x)}  \cdots  {{I_W(H_m \cap W)}_x}^{\alpha_{im}(x)}$$
for suitable exponents. We shall see that if $\cI$ is monomial,  then it can be  resolved rather easily, essentially in a combinatorial way.
\end{voi}

\begin{voi}
 \label{V:am2}
{\it The function $\Gamma $}. 
Let $\cI $  be  a monomial multi-ideal as in \ref{V:am1}. We define   
 the  function ${\Gamma} ={ \Gamma} _{\cI} $ from $S:=\sg (\cI)$ to ${\mathbb Z} \times {\mathbb Q} \times {\mathbb Z}^{\mathbb N}$ by setting, for 
  $z \in S$,  $  { \Gamma} (z) = (- {\Gamma }^{(1)} (z), {\Gamma }^{(2)} (z), {\Gamma }^{(3)} (z))$,  
where the values ${\Gamma }^{(1)}(z)$, ${\Gamma }^{(2)}(z)$ and ${\Gamma }^{(3)}(z)$ are obtained as follows.
 
\begin{itemize}
 \item[${\Gamma }^{(1)}$:] We set ${\Gamma }^{(1)} (z)=p$, where $p$ is the smallest integer $p$ such that there is an index   $q \in \{1,\ldots,n \}$ for which an expression like $(\star)$ is valid for $(I_q)_z$, and for indices $i_1, \ldots, i_p$ we have 
  $$ (1) \qquad \alpha_{qi_1} (z) + \cdots + {\alpha_{qi_p}(z)} \geq b_q~.$$ 
\item[${\Gamma }^{(2)}$:] Letting $p= {\Gamma }^{(1)} (z)$, 
 ${\Gamma }^{(2)} (z)$ is the maximum of the rational numbers 
$$({\alpha _{qi_1} (z)} + \cdots + {\alpha _{qi_p}(z)})/b \, ,$$
for  indices $i_1, \ldots, i_p,q$   for which an equality of type (1) holds.
\item[${\Gamma }^{(3)}$:] Consider the set of all sequences $(i_1, \ldots, i_p,0,0, \ldots)$ such that, for some index $q$,  
$({\alpha _{q i_1} (z)} + \cdots + {\alpha _{q i_p}(z}))/b_q = {\Gamma }^{(2)} (z)$, 
lexicographically ordered. Then,   
${\Gamma }^{(3)}(z) \in {\mathbb Z}^{\mathbb N}$ is the maximum in this ordered set.
\end{itemize}

\smallskip

The function $\Gamma $ has two important properties whose proofs,  very similar to those for basic objects found in  
 \cite[5]{EV}, \cite[6.4]{Cu}, or \cite[5]{BMF},  we omit.

 (a) When the target is lexicographically ordered, $\Gamma$ is an upper semi-continuous function.

(b) If $C=\ma {(\Gamma}_{\cI})=\{x \in \sg (\cI): {\Gamma}_{\cI}(x) = \max ({\Gamma}_{\cI})   \}$, 
 then $C$ is a permissible center for the multi-ideal $\cI$, called the {\it canonical monomial center}. The transform ${\cI}_1$ of $\cI$ is again  monomial, satisfying ${\rm max}~ (\Gamma _{{\cI}_1}) < {\rm max} ~(\Gamma_{\cI}) $ (see \cite[6.17]{Cu}). 

From this fact, it easily follows that if we iterate
 the process of transforming a monomial multi-ideal, using each time the canonical monomial center, after a finite number of steps, we reach a situation where the singular locus is empty. 

It is clear that $\Gamma _{\cI}$ is compatible with open restrictions in the sense that, if $U$ is open in $M$, then 
$({{\Gamma} _{\cI}})_{|{U \cap S}} = {\Gamma }_{\cI |U}$ (with $S=\sg \, (\cI)$).
\end{voi}

\begin{voi}
 \label{V:prop}
{\it Proper transforms}. Let $\cI=(M,W,(I_1,b_1), \ldots,(I_n,b_n),E=(H_1, \ldots, H_m))$ be a multi-ideal, $C$ a permisible center, $\cI \leftarrow \cI_1=
(M_1,W_1,(I_1[1],b_1), \ldots,(I_n[1],b_n),E_1)     $ the transformation with center $C$, $H$  the exceptional divisor, and $\cE=I_{W_1}(H \cap W_1)$. 

In \ref{V:par} we have defined the {\it proper transform} of $I_i$ 
for a fixed $i \in \{  1, \ldots, n \}$). This is the $W_1$-ideal 
${\bar I_i[1]}:={\cE}^{-c_{i1}}I_i{\cO}_{W_{1}}$,
  where the exponent $c_{i1}$, which depends on $z \in W_1$,  is defined as follows.  
  If $z \notin H$, then $c_{i1}=0$; if $z$ is in $H$ and its image $x$ in $W_s $ belongs to the irreducible component $C'$ of $C$,  then $c_{i1}=\nu_g(I_i)$, where 
 $g$ is the generic point of $C'$. Then we have:
$$ (1) \quad I_i[1]= \cE ^{-b_i}I_i{\cO}_{W_{1}}=
  \cE ^{-c_{i1}}(I_i{\cO}_{W_{1}})\cE ^{c_{i1}-b_i} = {\bar I_1}[1]\cE^{a_{i1}} \, , $$ 
with $a_{i1}=c_{i1}-b_i$.
\smallskip

More generally, given a permissible sequence of multi-ideals 
$$(2) \qquad \cI = {\cI}_0 \leftarrow \cdots \leftarrow {\cI}_s ~,$$
where ${\cI}_j=(M_j,W_j,(I_1[j],b_1), \ldots, (I_n[j],b_n),E_j)$,  
we may define the $W_j$-ideals ${\bar I_i}[j]$  inductively: ${\bar I_i}[0]:=I_i$  and, for $j>0$, 
${\bar I_i}[j]$ is the proper transform of ${\bar I_i}[j-1]$, $j=1, \ldots, s$. We call the $W_j$-ideal ${\bar I_i}[j]$ the 
{\it proper transform} of $I_i$ to $W_j$.

Then, for a fixed index $i=1, \ldots,n$, if $0 \le j \le s$ and $1 \le q \le j$,  
writing  
$E_j=(H_1, \ldots, H_m,H_{m+1},  \ldots, H_{m+j})$, and 
$\cE _q = {I_{W_j}}(H_{m+q }\cap W_j)$, 
we have an equality of $W_s$-ideals 
$$ (3) \qquad {I}_i [j]={\bar {I_i}} [j] \, \cE _1^{a_{i1}} \ldots   \cE _j^{a_{ij}} , $$ 
which is called the {\it proper factorization} of the ideal $I_i[j]$. 
 
Likewise, we may consider the proper factorization of the ideal ${I_i}[j+1]$ 
$$(4) \quad  {I}_i [j+1]={\bar {I_i}} [j+1] \, {{\cE '}_1}^{{\alpha}_{i1}} \ldots   {{\cE'} _j}^{\alpha_{ij}} 
{{\cE'} _{j+1}}^{\alpha_{i,j+1}} {\cO}_{j+1}  ,$$
 where $\cE'_q = I_{W_{j+1}}(H_{m+q}  \cap W_{j+1})$, with  
$H_q \subseteq M_{j+1} $ still denoting the strict transform of 
$H_q \subseteq M_{j} $ to $M_{j+1}$, $q=m+1, \ldots , m+j$, and $H_{m+j+1}:=H$  being the last exceptional divisor. 
 
The exponents $a_j$ and $\alpha _j$ that appear respectively in (3) and (4) are related. Indeed,    a calculation similar to that of (1) shows that
$a_{iq}={\alpha}_{iq}$, $q=1, \ldots, j$,  and 
$${\alpha}_{i,j+1}=c_{i,j+1}-b_i+c_{i1} + \cdots + c_{ij} \, ,$$
where the numbers $c_{i1}, \ldots,c_{i,j+1}$, which depend on $z \in \sg(W_{j+1})$,  are as follows. If $z \notin H$ (the exceptional divisor), all of them are zero. Actually, this case is trivial since if  $z \notin H$, then 
${(\cE ' _{j+1})}_z = {\cO _{W_{j+1},z}}$.  If $z \in H$, let  its  image of $x  \in W_{j}$ belong to   the irreducible component $C'$ of the center $C_{j}$ (used in the transformation $\cI_{j} \leftarrow \cI_{j+1}$), and let $g$ be the generic point of $C'$.  Then 
$c_{iq}(z) = \nu_g(\cE _q)$, $q=1, \ldots, j$, and 
$c_{i,{j+1}}(z) = \nu_g({\bar I_i}[j])$.  
\end{voi}

\begin{voi}
\label{V:am3} {\it The functions  ${\o _i}[j]$ and ${ \o}[j]$}.  
Consider a permissible sequence as in \ref{V:prop} (2). 
  For a fixed index $i \in \{ 1, \ldots, n\}$ and $x \in \sg (\cI _j)$, define 
$${\o _i}[j](x):= \nu _x ({\bar{I_i}}[j])\,/\,b_j \, ,$$
where   ${\bar{I_i}}[j])$ denotes the proper transform of $I_i$ to $W_j$. 

 Now we introduce a function ${\o }[j]   :\sg(\cI _j) \to \mathbb Q$ by the formula: for 
 $x \in \sg(\cI _j)$,   
 $${ \o}[j](x)= \min \{ \o _1 [j](x), \ldots,  \o _n [j](x) \}  \,.$$

The functions  $\o _i[j]$ and  $\o[j]$ are upper-semi-continuous,  from $\sg (\cI _j)$ to ${\mathbb Q}$. 

 We say that the sequence \ref{V:prop} (2) is a ${ \o}$-{\it sequence} if the involved permissible centers  $C_1, \ldots, C_{s-1}$  satisfy 
  $C_j \subseteq \ma({  \o} [j])$ (the set of points where ${ \o}[j]$ reaches its maximum value).
\end{voi}

\begin{voi}
\label{V:ensi}
{\it A comment of the function $\o [s]$}.  In the notation of \ref{V:am3}, let $x \in \sg(\cI _s)$ and ${ \o}[s](x) = \beta$. Since  ${ \o}[s](x)= \min \{ \o _1 [s](x), \ldots,  \o _n [s](x) \}$, for some index $v$ we have 
$     \beta = {\o}[s](x)=  \o _v [s](x)  $. However, it does not follow that for a suitable open set $U \subset M_s$ we  have  
$    { \o}[s](z)=  \o _v [s](z)      $ for all $z \in \sg(\cI) \cap U$. Here is an example.

\smallskip

{\it Example}. Let $\cI = (M,W,(I_1,b_1),(I_2,b_2), \emptyset)$, where $M=W={\mathbb A}^2=\Spec (  R   )$, $ R= {\mathbb C}[x,y])$, $I_1=(x^2 y^4)R$, $b_1=2$, $I_2=( x^4 y  )R$, $b_2=3$.  Then, 
$S:=\sg(\cI)=V(x)$ (the $x$-axis). 

Write ${ \o}= { \o}[0]$, ${ \o}_1={ \o}_1[0]$, ${\o}_2= {\o}_2[0]$, and   $a_t=(0,t) \in S$, $t \in \mathbb C$. 
 Then 
$\o _1(a_0)=6/2=3$,  $\o _2(a_0)=5/3$,  hence ${ \o}(a_0)=\o _2(a_0)=5/3$. If $t \not= 0$, 
   $\o _1(a_t)=2/2=1$,  $\o _2(a_t)=4/3$,  hence ${ \o}(a_t)=\o _1(a_t)=1$. Thus, although ${ \o}(a_0)=\o _2(a_0)$, no neighborhood $U$ of $a_0$  satisfies ${ \o}(z)=\o _2(z)$ for all $z \in S \cap U$, because such a set $U$  contains points $a_t$, $t \not= 0$. In this example, $\max({ \o})=5/3$ and $\ma ({ \o})= \{  a_0 \}$.
\end{voi}

\begin{pro}
\label{P:nosube}
If  sequence (2) of \ref{V:prop} is a  ${ \o}$-sequence,  then  
$ \max ({    \o }[j+1]) \le  \max ({  \o }[j]) $.
\end{pro}
\begin{proof}
See \cite[6.4]{Cu}.
\end{proof}

\begin{voi}
\label{V:am4} {\it The functions  ${ t }[s]$}. Assume the sequence \ref{V:prop} (2) is a ${ \o}$-sequence. For $0 \le i \le n$, let $q$ be the smallest index such that 
$\max ({ { \o}}[s-1])  >  \max ({{ \o} }[s])$
 (recall that 
$\max ({{ \o}}[q]) \ge \max ({{ \o }}[q+1]) \ge \cdots $). 
Let  $E_s ^{-}$ denote the set of hypersurfaces in $E_s$, which are strict transforms of hypersurfaces in $E_{q}$.    
 For $x \in \sg (\cI _s)$, 
 write 
$n[s](x)$ for the number of hypersurfaces in ${E_s}^{-}$ that contain the point $x$. Now, for such $x$, define  
$$ { t}[s](x):=({ \o }[s](x),n_s(x) )\,.$$
These are upper-semi-continuous functions from $\sg(\cI_s)$ to ${\mathbb Q} \times {\mathbb N}$. We denote the set of points where 
$ {t} [s]$ reaches its maximum value $\max( { t} [s])$ by   $\ma( { t} [s])$. 

The sequence \ref{V:am3} (1) is called a $ {t}$-sequence if each center $C_j$ involved satisfies 
$C_j \subseteq \ma({ t}[j])$, $i=0, \ldots, s-1$. It can be proved, essentially as in \cite{EV}, that if  
\ref{V:prop} (2) is a $ {t}$-sequence,  then 
$ {\max( { t} [j+1]) \le  \max({t} [j]),} ~ j =0, \ldots, s$, for all $i$.
\end{voi}

\begin{voi}
 \label{V:am5}
  We shall see that, 
assuming resolution for locally nice multi-ideals of dimension $d-1$ and 
starting from a marked ideal $\cI$ of dimension $d$, there is  a ${t} $-sequence which leads to a marked ideal $\cI_s$ where $\max ({ t}[s])$ has dropped.  
To produce such a sequence we shall use  certain auxiliary multi-ideals, which are locally nice and whose singular set locally coincides with $\ma ( { t}[j])$. Since these  auxiliary multi-ideals are locally nice the inductive results of \ref{V:ut5} apply to them.     The construction of these multi-ideals, based on similar work in the context of basic objects presented in \cite{EV},  will be discussed in the remainder of this section.
\end{voi}

\begin{voi}
 \label{V:I'}
{\it The multi-ideal ${\cI}' _{s}$}.  
Consider a $\o$-sequence of basic objects 
$$(1) \qquad {\cI}_0 \leftarrow \cdots \leftarrow {\cI}_s \,.$$
We write ${\cI}_j=(M_j,W_j, (I[j],b),E_j)$ and 
 express the maximum value of its function $ ({ \o}[s])$ as a fraction $b[s] / b$. We 
    associate 
the multi-ideal 
$$  {\cI}_s ' =(M_s,W_s,    (I[s],b), ({\bar I} [s], b[s]), E_s)$$
to the basic object ${\cI}_s$.
\end{voi}
 We describe some properties of the multi-ideal $\cI'_s$ in the next two propositions, where we keep the assumptions and notation of \ref{V:I'}.
\begin{pro}
 \label{P:sini}  
$  \ma ({ \o}[s]) = \sg({\cI}'_{s})     $
\end{pro}

\begin{proof} 
 If $x \in W_s$ is in $ \sg({\cI}'_{s})$, then $x \in \sg (I[s],b) = \sg ({\cI}_s)$, i.e., the domain of $\o [s]$. Moreover, $x \in 
\sg ({\bar I}[s],b[s])$, i.e., $\nu_x({\bar I}[s]) \ge b[s]$. By maximality, $\o [s](x) = \nu_x({\bar I}[s]) = b[s]$, i.e., 
$x \in \ma (\o[s])$. For the  inclusion $\subseteq$, if $x \in \ma ({ \o}[s])$, then $x \in \sg({\cI}_s) = \sg (I[s],b)$ and 
$\o [s](x) = \nu_x({\bar I}[s]) = b[s]$. Hence $x \in \sg ({\bar I}[s],b[s])$; ie., $x \in \sg({\cI}'_{s})$
\end{proof}
Notice that $ b[s]$ may be $  \ge  b$, or not. If $b[s]  \ge  b$  then, for $x \in W_s$,  $\nu_x({\bar I}[s]) \ge b[s]$ implies 
 $\nu_x({I}[s]) \ge b$. In this case, the singular set of the basic object 
$(M_s,W_s,    ({\bar I} [s], b[s]), E_s)$  agrees with $\ma (\o [s])$. 
\begin{pro}
 \label{P:pro'}  
$(i)$ A closed subscheme $C$ of $W_s$ is a permissible center for $\cI' _s$ 
 if and only if $C \subseteq F= \ma ({\o}[s]) $.     

$(ii)$  Let $C$    be a ${\cI} _s$-permissible center contained in $\ma (\o[s])$             (which, by ($i$) $C$, is also an ${\cI}' _{s}$-center),  and consider 
 the transformations   ${{\cI}_s \leftarrow {\cI}_{s+1}}$ and 
$\cI ' _s \leftarrow {[\cI' _{s}]}_1$, with center $C$. 
  Then we have:
 
(a) If  $\max({\o} [s+1])  > \max({ \o} [s])$, then, $\sg ({[\cI ' _{s}]}_1)=\emptyset$. 

(b) If  $\max({ \o} [s])=({ \o} [s+1])$,
 then $\cI' _{s+1} = {[\cI ' _{s}]}_1$ 
\end{pro}
\begin{proof}
 $(i)$ is a consequence of \ref{P:sini} and of the fact that for both $\cI ' _s$ and $\cI _s$, the sequence of hypersurfaces is the same, namely $E_s$ (\ref{V:I'}).

\smallskip

$(ii)$ Since  $b _i[s] \ge b_i[s+1]$), only  (a) or (b) above may occur.   

Assume we have (a). One of the pairs in ${[\cI]}_1$ is the transform of $\uP=({\bar I}[s],b[s])$, namely 
$\uP _{1} = (\cE ^{-b[s]}{\bar I}_s {\cO}_{W_{s+1}}, b[s])$, with $\cE$ defining the exceptional divisor. So, the ideal 
$  \cE ^{-b[s]}{\bar I}_s {\cO}_{W_{s+1}}  $ is the controlled transform of ${\bar I}[s]$. 
Since $\uP$ is good, this controlled transform agrees with the proper one, that is $\uP _1=(  {\bar I}[s+1], b[s] ) $. By definition of $b[s+1]$,  
$\nu _x ({\bar I}_{s+1})=b[s+1] < b[s]$, for all $x \in \sg(I[s+1]$. So, the inequality $\nu _x ({\bar I}[s+1]) \ge b[s+1]$ cannot hold for any $x \in W_{s+1}$, which means: $\sg({\cI}_{s+1})=\emptyset$.

Now assume (b), i.e., 
$\max ({ \o}[s]) =  \max ({\o}[s+1])=\gamma$. This means  $b _i[s] =  b_i[s+1]:=b'$.
 To show that $\ma ({ \o}[s+1]) =  \sg( {[\cI' _{s}]}_1  )$, it suffices to verify that the pair 
$({\bar I}_i[s+1],b') $ is the controlled transform of the pair $({\bar I}_i[s],b')$. By definition, 
$({\bar I}_i[s+1],b')$ is the proper transform of $({\bar I}_i[s],b')$. But by the definition of $b'=b[s]$, the pair$({\bar I}_i[s],b')$ is good, thus its proper and controlled transforms agree (\ref{V:par}).   
\end{proof}
Next we define the auxiliary multi-ideal announced in \ref{V:am5}.
\begin{voi} 
\label{V:dia} {\it The multi-ideal ${\cI}^{\diamond}_{s}$}.  We  assume that  \ref{V:I'} (1) is a $t$-sequence of basic objects, and we write     $E_j=(H_1,  \ldots, H_{m+j})$, $j=0, \ldots, s$.

Note that \ref{V:I'} (1) is also a $ \o$-sequence, and consequently, 
 $ \max({ \o }[j]  ) \ge  {\max ( { \o}[j+1] )} $ for all $j$.
 Let $q$ be the smallest index such that 
$\max ( {\o }[j]  )=\max ( { \o}[s]  )$ (so either $q=0$ or $\max ( { \o}[q-1]>\max ( { \o}[q]  )  )$. Let   
${ E}_s ^{-}:=(H_1,\ldots, H_q) \subset E_s $ (i.e., $E_s$ consists of the strict transforms of hypersurfaces in $E_q$) and 
 ${ E}_s ^{+}={ E}_s  \setminus E_s ^{-}$.

Write 
  $\max ({ t}[s])=(\gamma [s], \tilde n[s]) \in \mathbb Q \times \mathbb Z$, where 
$\gamma [s]=\max ({ \o} [s])=b[s]/b$. Let $T[s]$ denote the set of subsequences $(H_{j_1}, \ldots, H_{j_{\bar n}})$ of $E^{-}_{i,s}$, $j_1 < \cdots  <  j_{\bar n}$ (see \ref{V:am4}), and let
$L[s] = \prod (I_{W_s}( H_{j_1} \cap W_s) + \cdots +  I_{W_s}( H_{j_{\bar n}} \cap W_s )) \, ,$
where the product is taken over all sequences $(H_{j_1}, \ldots, H_{j_{\bar n}}) \in T[s]$.

 We  associate to $\cI_s$ the multi-ideal:  
$${\cI}^{\diamond}_{s} = (M_s, W_s,          (I[s],b), ({\bar I} [s], b[s]) ,  (L[s],1) , { E}^{+}_{s}|U) \, .$$ 
\end{voi}

\begin{pro}
 \label{P:idiam} 
Using    the notation of \ref{V:dia},
             the multi-ideal  $\cI^{\diamond}_{s}$  has the following properties: 
\begin{itemize}
\item[(i)] It is locally $2$-nice (i.e., $l$-nice, for  $l=2$). 
\item[(ii)] A subscheme $C$ of $W_s$ is a $t$-permissible center for $\cI_s$ if and only if $C$ is a permissible $\cI^{\diamond}_{s}$-center.
\item[(iii)] Let $C$ be a $\cI^{\diamond}_{s}$-permissible center.  
 If ${\cI \leftarrow {\cI}_{s+1}}$ and 
${\cI^{\diamond}_{s} \leftarrow {[\cI^{\diamond}_{s}]}_1}$ are the transformations  with center $C$,  then, $\max({ t}  [s])  \ge  \max({ t} [s+1])$. Moreover, 
 \begin{itemize}
\item[($\alpha$)]  if 
 $\max({ t}  [s]) > \max({ t} [s+1])$, then $\sg({[\cI ^{\diamond}_{s}]}_1)=\emptyset$;  
\item[($\beta$)] if 
 $\max({ t}  [s])=\max({ t} [s+1])$, then
$\ma({\bar t}[s+1]) = \sg({[\cI ^{\diamond}_{s}]}_1) \, .$ 
\end{itemize}
\end{itemize}
\end{pro}

\begin{proof} 
(i)  The pair $({\bar I}[s],b[s])$ appears in $\cI^{\diamond}_{i,s}$, namely in the second position. 
 Hence we have $\sg(\cI^{\diamond}_{s}) \subseteq \sg({\bar I}[s],b[s])$. 
 For any point $x \in \sg(\cI^{\diamond}_{s})$, 
since $b[s] = \nu _x ( {\bar I}[s])$,  there is an element $g \in  {{\bar I}[s]}_x$ such that 
$\nu _x (g) = b [s]$. 
 Hence taking suitable derivatives, the stalk 
${\Delta ^{b[s]}}{({\bar I}[s])}_x$ contains an element of order one, which defines a regular hypersurface $Z$ containing $x$
 on a appropriate neighborhood $U$ of $x$. From the definition of $E^+_{s}$,  it follows that $Z$ is transversal to 
$E^+_{s}|U$ and, therefore, $\cI^{\diamond}_{s}$ is locally $2$-nice.

(ii)  
We have $  \sg(\cI^{\diamond}_{s})=\sg({\cI}_s') \cap \sg (L[s],1)   $. 
 We claim that
\begin{itemize}
 \item [($\star$)]
 a subscheme $C \subseteq \ma(t[s])=\sg(\cI^{\diamond}_{s})$ has normal crossings with $E_s$ (the set of hypersurfaces of $\cI _s) $ if and only if $C$ has normal crossings with $E_s^{+}$(the set of hypersurfaces for $\cI^{\diamond}_{s}$). 
\end{itemize}
First we prove ($\star$) in case   
${ \o}[s-1] > { \o}[s]$. To do so,  let 
   $\max (t_i [s-1]) = (\beta, \tilde n)$. Assume that $x \in \sg ({\cI ^{\diamond}}_s)$ is a closed point,  and that 	
  $H_{i_1}, \ldots, H_{i_{\tilde n}}$ are the hypersurfaces of ${ E} _s={ { E}_s}^{-}$  containing $x$ (in this case $  {{ E}_s}^{+} = \emptyset$). 
 From the maximality of $\tilde n
 $,  if we restrict to a suitable neighborhood of $x$, $C \subseteq  H_{i_1} \cap \ldots  \cap H_{i_{\tilde n}}$. 
 This inclusion easily implies that 
 $C$ must have normal crossings with ${\tilde E}_s$. 
  The general case follows from the situation just discussed and the fact that the hypersurfaces in ${E}_s ^{-}$ are the exceptional divisors that appear when we go from $\cI_q$ to $\cI_s$
  (in the notation of \ref{V:am4}).

Now we prove that if $C$ is a $t$-center for  $\cI _s$, then it is a $\cI ^{\diamond} _s$-center. Since $C$ is also a $\o$-center for $\cI _s$,  
 by \ref{P:pro'}, we have $C \subseteq \sg({\cI '}_s)$. 
 But by the assumed $t$-permissibility, we also have $\nu _x (L[s]) \ge 1$, i.e.,  $C \subseteq \sg ( L[s],1  )$ for all $i$. Thus 
$C \subseteq \sg(\cI^{\diamond}_{s})$ and, by ($\star$),  $C$  is a 
$\cI^{\diamond}_{s}$-center.

For the converse, $C$ being a $\cI^{\diamond}_{s}$-center implies $C \subseteq \sg(L[s],1)$ and, by \ref{P:pro'}, also  
$C \subseteq \sg( {\cI} ^{\diamond}_s)$. Thus, by ($\star$), we have the desired implication.

\smallskip

(iii) In \ref{V:am4} we have seen that  
$ \max(t[s])   \ge \max(t[s+1])$. Note that  we have the equality 
$$(1) \qquad         \max(t[s])   =  \max(t[s+1])     $$
if and only if  
$\max(\o[s])   =  \max(\o[s+1])$ and 
$   \max(n[s])   =  \max(n[s+1])      $. 

 {\it Case ($\alpha$}). Assume that   $\max(t[s])   > \max(t[s+1])$. If this happens because 
$\max (\o[s])   >  \max(  \o[s+1])$, then \ref{P:pro'} (a) shows that 
$\sg ( {[\cI ^{\diamond}]}_1   )=\emptyset$. If the reason is that 
$   \max(n_i[s])   >  \max(n_i[s+1])  $ then, letting  
 ${\uL}'$ denote the   controlled transform of $\uL:=(L_i[s],1)$, we have  
$ \sg ( {\uL}' )=\emptyset  $. So, in either situation 
$\sg({[\cI ^{\diamond}_{s}]}_1)=\emptyset$, and thus 
case ($\alpha$) is checked.

\smallskip

{\it Case ($\beta$)}. The statement follows from (ii) (b) of    \ref{P:pro'} and the fact that 
the equality  $ \max(n_i[s])   =  \max(n_i[s+1])   $  insures that $(L[s+1],1)$ is the controlled transform of $((L[s],1)$.
\end{proof}

\begin{voi}
\label{V:apli}
We explain the use of the multi-ideal       $\cI^{\diamond}_{s}$ thus expanding \ref{V:am5}. 

Assume we have a ${ t}$-permissible sequence of basic objects 
$$(1) \qquad    
\cI _0 \leftarrow \cI _{1} \leftarrow  \cdots \leftarrow  \cI _{s} \, ,$$ 
where 
${ \o}[s-1] > { \o}[s]$. Consider the multi-ideal $\cI^{\diamond}_{s}$ of \ref{V:dia}, to which we may apply  the results of 
\ref{V:ut4} and  \ref{V:ut5},   since it is locally nice. 

Suppose the $\dim (\cI)=d$ and  $\dim (\ma (t[s])) \le d-2$.. 
Then, by \ref{P:idiam}, we  get a collection of open sets ${\mathcal U}$ covering $\ma (t[s])$, so that for each set $U \in {\mathcal U}$ we have an adapted hypersurface $Z_U$ and an inductive $(d-1)$-dimensional multi-ideal $ {(\cI^{\diamond}_{s}|U)}_{Z_U}  $. By the results in \ref{V:ut4} and  \ref{V:ut5} and by 
  induction on the dimension, we obtain a resolution 
$$\cI^{\diamond}_{s}  \leftarrow    {[\cI^{\diamond}_{s}]}_1   \leftarrow     \cdots    \leftarrow    {[\cI^{\diamond}_{s}]}_{r_1} .   $$ 
By \ref{P:idiam}, this resolution induces a $t$-permissible sequence   
$\cI _s \leftarrow \cI _{s+1} \leftarrow \cI _{s+r_1}$, extending  
 (1), so that  
 $\max ({ t}[s]) = \ldots = \max ({ t}[s+r_1-1])$. Concerning $\cI _{s+r_1}$, either $\sg(\cI _{s+r_1})=\emptyset$ or 
  ${ t}[s+r_1-1] > {t}[s+r_1]   $. If the second possibility occurs,     we take the associated multi-ideal $ {\cI}^{\diamond}_{s+r_1} $ and repeat the process just described.      Iterating further, if necessary, we see that  there is an index  
 $r' \ge r_1$ such that either 
 $ \sg (   {\cI}_{s+r'})=\emptyset $ or, for some index $i$,  $\max ( \o  _i [s+r']  )= 0$. Thus, either we have a resolution of $\cI _0$ or we reached a   situation where  the basic object 
  $\cI _{s+r'}$ is monomial. Then this can be resolved as indicated in \ref{V:am2}. 

The above explanation on the use of $\cI^{\diamond}_{s} $   will be formalized  in the next section, where we describe our resolution algorithm.
\end{voi}

\begin{voi}
\label{V:cone}
The assumption on the dimension of  $  \ma (t[s])  $ made in \ref{V:apli} is not really restrictive, because 
  the situation where   the codimension of $  \ma (t[s])  $ in $W_s$ is equal to one is easily handled. Indeed, if $C$ is the union of the one-codimensional components of $\ma (t[s])$ then, as  a   consequence of Propositions \ref{P:heq}  and 
\ref{P:idiam}, the variety 
$C$ is a permissible center for $\cI _s$,  and if $\cI _{s+1}$ is the transform of $\cI _s$ with center $C$, $\ma (t[s+1])$ has no irreducible components of codimension one. See    \cite[3.9]{T2} for more details.
\end{voi}

\begin{rem}
\label{R:tga}
 {\it The invariance of the functions $t$ and $\Gamma$}. In \cite[2.8]{T2}, it is proved that when working with basic objects, the $t$ and $\Gamma$ functions of \ref{V:am2}
 and \ref{V:am4} respectively, are invariant under equivalence. 

Let us explain more precisely the meaning of this statement for  $t$. Assume  $\cI=(M,W,(I_1,b), E)$ and 
 $\cJ=(M,V,(J,c),E)$  are equivalent basic objects of the same dimension,  and that  
 $$(1) \quad    \cI=\cI _0 \leftarrow \cI _{1} \leftarrow  \cdots \leftarrow  \cI _{s} ~ {\mathrm {and}}~
(2) \quad    \cJ=\cJ _0 \leftarrow \cJ_{1} \leftarrow  \cdots \leftarrow  \cJ _{s}$$
\noindent are $t$-sequences (\ref{V:am4}), 
where the center used for corresponding transformation is the same in both cases. Then,   indicating by  $t[j]$ and   $t'[j]$) the $j$-th $t$-function   of (1) and (2) respectively, we have 
 $t [j]=t' [j]$  for all $j$. 

The proof is easily obtained using the following Hironaka's theorem  (\cite{Hiro},  \cite[6.1]{BMF}):  
 if  $\cI=(M,W,(I,b), E)$ and 
 $\cJ=(M,W,(J,c),E)$  are equivalent basic objects with $\dim (W)=\dim (V)$, then $\nu_x(I)/b = \nu_x(J)/c$, for all $x \in \sg(\cI)=\sg(\cJ)$. 

Now we explain the statement about $\Gamma$. Assume that in the sequences (1) and (2) we have that $\max (\o[s])=0$, and because  $\cI \sim \cJ$, also 
$\max (\o'[s])=0$, where $\o'[s]$ is the $s$-th $\o$-function of $\cJ$. Then both $\cI_s$ and $\cJ _s$ are monomial and  
$\Gamma _{\cI_s}= \Gamma _{\cJ _s}$. Thus $\cI_s$ and $\cJ _s$ have 
 a common  canonical monomial center $C$. We assert that  if $\cI_{s+1}$ and $\cJ _{s+1}$ are the transforms of $\cI_s$ and $\cJ _s$ respectively with center $C$, then again 
$\Gamma _{\cI_{s+1}}= \Gamma _{\cJ _{s+1}}$, and so on.
\end{rem}
These results may be extended to the case where $\cI$ and $\cJ$ are multi-ideals. A simple way to verify this fact is to prove that if $\cI$ is multi-ideal, its $t$ and $\Gamma$ functions coincide with those of an associated basic object    $B_{\cI}$     (\ref{R:getasu}) and then use the result just stated. Since we will not use this fact, we omit  the details.

\begin{pro}
\label{P:ev}
Suppose $\cI=(M,W,(I,b),E)$ and $\cJ=(M,V,(J,c),E)$ are equivalent basic objects of the same dimension,  and assume we have $t$-sequences as (1) and (2) of \ref{R:tga}. 
  Then, the auxiliary multi-ideals 
${{\cI}_s^{\diamond}}$ and  ${{\cI}_s^{\diamond}}$ 
( \ref{V:dia}) are equivalent.
\end{pro}
The proof is practically the same as that of Proposition 3.10 of \cite{T2}. It is based on the observation that permissible sequences 
 ${{\cI}_s^{\diamond}} \leftarrow {[ {{\cI}_s^{\diamond}}]}_1 \leftarrow \cdots $
correspond to $t$-sequences extending \ref{R:tga} (1), with equalities $\sg({[{{\cI}^{\diamond}}_s]_j})=\ma (\cI _{s+j})$ (and similarly for ${\cJ}^{\diamond}_s$), as well as on the fact that the $t$-functions of equivalent objects coincide (\ref{R:tga}).

\section{An algorithm}
\label{S:C}

In this section, using the concepts developed is sections \ref{S:tools} and \ref{S:A},  we prove the following result:

\begin{thm}
\label{T:funcan}
To each nonzero unresolved  multi-ideal $    \cI = (M,W,(I_1,b_1), \ldots, (I_n,b_n),E)    $  
we may attach  algorithmic resolution  functions $h_i$, $i=0, \ldots, r-1$, where $r \ge 1$ depends on $\cI$ (\ref{V:alre}). This process satisfies conditions (a) and (b) of  
\ref{V:alre} and also the following  condition: 
\begin{itemize}
 \item [(c)] If $\cI = \cI _0 \leftarrow \cdots \leftarrow \cI _r$ is the resolution 
determined by the functions $h_i$ (i.e., the $j$-th center is $C_j:=\ma (h_j)$), then there is an index $s$, $0 < s \le r$, such that 
 the induced sequence 
$\cI = \cI _0 \leftarrow \cdots \leftarrow \cI _s$
 is a $t$-sequence (\ref{V:am4}); while 
  for all $j > s$ the multi-ideal $\cI _j$  is monomial and $C_j$ is the monomial canonical center 
 (\ref{V:am2}).       
\end{itemize}
\end{thm}

\begin{voi}
\label{V:anu}
The proof of the above Theorem is very similar to that of Theorem 4.1 in \cite{T2}. For that reason, in the remainder of this section, we  describe the algorithm but  we omit several verifications, referring instead to \cite{T2}. 

For each positive integer $d$,  we   introduce a totally ordered set
$\Lambda^{(d)}$  and, for each $d$-dimensional multi-ideal $\cI$, we introduce functions
$h_i$ (with domains as in Definition  \ref{V:alre} and  with values in $\Lambda^{(d)}$)  
that are the  resolution functions of our algorithm.

We discuss separately the cases $d=1$ and $d$ arbitrary, where  $d= \dim (W)$. In the sequel, $\cS _1:= {\mathbb Q} \times  {\mathbb N }$ and  
$\cS _2:= {\mathbb Q}  \times {\mathbb Z} \times {\mathbb  Z} ^ {\mathbb N}$. 
\end{voi}

\begin{voi}
\label{V:dem1}
 {\it  The functions $h_i$ when $d=1$}. If $\dim \,(\cI_0) = 1$, we set  $\Lambda ^{(1)}= \cS _1 \cup \cS _2 \cup  \{\infty _1\}$, where if $a \in \cS _2 $ and $b \in \cS _1 $, then $a >b$  and $\infty _1$ is the largest element of the set. First, we define  for 
$x \in \sg (\cI_0)$, $h_0(x)= t(x)$.  Next, if $h_i$ is defined for $i < s$, determining a $t$-permissible sequence 
$\cI_0\leftarrow {\cI}_1 \leftarrow \cdots \leftarrow \cI_s$, then we set, for $x \in \sg(\cI _s)$, $h_s(x)=t[s](x)$, if ${\o}[s](x) > 0$, while we set $h_s (x)=\Gamma _{{\cI}_s} (x)$ in case ${\o}[s](x) = 0$. Since in this one-dimensional situation   $C_j=\ma (h_j)$ is always a finite collection of closed points (for all $j$), it follows that these are permissible centers.

The proof that the above are resolution functions and conditions (a), (b) and (c) in  \ref{T:funcan} are valid can be found in \cite[4.2]{T2}.
\end{voi}  
\begin{voi}
\label{V:dem2}
{\it The functions $h_i$ in general}. 
 Now assuming that algorithmic resolution functions satisfying (a), (b) and (c) are available when $\dim (W) < d $,  we define resolution functions $h_j$ for multi-ideals of dimension $d$. If $d >1$,   
the totally ordered set of values will be 
 $\Lambda ^{(d)}={(\cS _1 \times {\Lambda}^{(d-1)})}\cup \cS _2 \cup \{ \infty  _d\}$,
  where 
 $\cS _1 \times {\Lambda}^{(d-1)}$ is lexicographically ordered, any element of $\cS _2$ is larger than any element of $\cS _1 \times {\Lambda}^{(d-1)}$, and 
 $\infty _d$ is the largest element of $\Lambda ^{(d)}$.

 We shall deal first with the case where $\cI$ is a $d$-dimensional marked ideal or basic object, i.e., in the notation of Definition \ref{D:multi-ideal}, $n=1$. But notice that our inductive hypothesis is that  algorithmic resolution  functions satisfying (a), (b) and (c) are defined not just for basic objects but for {\it multi-ideals} of dimension $<d$. 
 
Given a nonzero marked ideal, or basic object,  $\cI_0=(M_0,W_0,(I_0,b),E_0)$ of dimension $d$,   we shall describe first the corresponding resolution function $h_0$.  So, if $x \in \sg({\cI}_0)$, we must give the value $h_0(x)$. We necessarily  have $\o [0] (x) > 0$. Let $N_1$ be the union of the 1-codimensional components of $\ma \, (t[0])$ with  two cases: ($i$)  $N_1 \not= \emptyset $ ($ii$) $N_1 = \emptyset $.   

In case ($i$), set $h_{0}(x)= {\infty}_d$ if $x \in N_1$, and set $h_0(x)=(t[0])(x),\infty _{d-1})$ otherwise.  So, the 0-th center $C_0$ is $N_1$. By \ref{V:cone}, $C_0=N_1$ is a permissible center.  

In case ($ii$), pick up an open neighborhood $U$ of $x$ such that the restriction ${{\cI}^{\diamond}_{0}}|U$ is nice (see \ref{V:dia}), and let $Z$ be an
adapted hypersurface.
 Consider the inductive multi-ideal ${\cI ^*}_{Z}:={(    {{\cI}^{\diamond}_{ 0}}|U     )}_{Z}$,  which is  nonzero multi-ideal by our assumption on $N_1$. By induction on the dimension,  resolution functions $\widetilde{h_{Z,j}}$ are defined for the multi-ideal $\cI^*_{Z}$. 

Then set 
$h_0(x):=(t[0](x),\widetilde{h_{Z,0}}(x))$.
   We claim that if  a different  open set and adapted hypersurface were chosen,   the result would be the same. First of all, since $t[0]$ and,  by induction,  the resolution functions 
  ${h_{Z,0}}$
  are compatible with restrictions to open sets, we may assume that the open set $U$ is the same in both cases. Let $Z'$ be the new adapted hypersurface. Now, by \ref{P:ev}, ${\cI^*}_{Z}  {\sim} {\cI^*}_{Z'}$. Since 
$\dim (Z)=\dim(Z')<d$, by induction, (b) is satisfied and
 $\widetilde{h_{Z,0}}(x)=\widetilde{h_{Z',0}}(x)$. So, 
 the value $h_0(x)$ is independent of the choices, and  the function  $h_0$ is well defined. 

\end{voi}

\begin{voi}
\label{V:dem2.1}
Suppose now that the  resolutions functions $h_i$, $i=0, \ldots, j-1$, satisfying (a), (b)  and (c)  of \ref{T:funcan}   have been defined, determining  centers $C_i = \ma \,({h_i}), i=0, \ldots, j-1$, and leading to a permissible sequence of basic objects:
 $$(1) \quad {\cI}_0 \leftarrow \cdots \leftarrow {\cI}_j, ~ {\cI}_i=(M_i,W_i,I_i,b,E_i), ~i=0, \ldots, j,~ j \ge 0 ~.$$
 We assume that if ${\cI}_{j-1}$ is not a monomial object, then (1) is a $t$-sequence. 
 Two cases are possible: ($\alpha$) $\max ({\o }[j]) =0$, and ($\beta$) $\max ({\o }[j]) > 0$. 
 
In case ($\alpha$), ${\cI}[j]$ is monomial. For $x \in \sg({\cI}_j)$,  let $\Gamma _{\cI _j}$ be its $v$-th $\Gamma$-function and set $h_j(x):=
\Gamma _{\cI _j } (x)$. 

In case ($\beta$), letting $N_1(j)$ denote the union of the one-codimensional components of $\ma (t[j])$, there are two sub-cases: ($\beta _1$)    
$N_1(j) \not= \emptyset$ and (${\beta}_2$) $N_1(j) = \emptyset$. 

In case (${\beta}_1$), set $g_j(x)=\infty _d$ if    $x \in N_1(j)$,  and set $g_j(x) = (t_j(x), \infty _{d})$ if $x \in \sg({\cI}_j)$ but  $x \notin  N_1(j)$.
  In case   (${\beta}_2$), which might be called {\it  the inductive situation}, if 
$x \in \sg({\cI}_j) \setminus  \ma (t[j])$, set $h_s(x)=(t[j](x), {\infty}_{d})$. If 
$x \in \ma (t[j])$, let $s$ be the smallest index such that $\max(t[s])=\max(t[j])$, and $N_1(s)$, the union of one-codimensional components of 
$\ma (t[s])$, be   empty. Let $x_s$ be the image of $x$ in $W_s$. Proceeding as in case $j=0$ (\ref{V:dem2}), we       pick up an  open neighborhood $U$ of $x_s$
  such that 
  $    {\cI _s ^*}_{Z_s}:=        {{\cI _s}^{\diamond}}|U$ is nice, 
 with adapted hypersurface $Z_s \subset U$.  We consider the inductive, necessarily nonzero, multi-ideal 
 ${\cI _s}^*_{Z_s}$ of dimension $d-1$. By induction on the dimension,  ${\cI _s}^*_{Z_s}$  has resolution functions $\widetilde{h_{Z_s,q}}$, $q=s, \ldots, s'$, where necessarily  $s' \ge j$. Set $h_j(x):=(t[j](x),\widetilde{h_{Z_s,j}}(x))$.

The final result is not affected by a different choice of the  open set  $U$ or the adapted hypersurface.   The proof is similar to that of case ($ii$) in 
  \ref{V:dem2}. Details may be found in 
\cite[4.4]{T2}.
\end{voi}
\begin{voi}
\label{V:dem4}
The process to successively produce the functions 
 described above terminates, and $h_0,h_1, \ldots$ are algorithmic resolution functions for the basic object $\cI$. Indeed,  in the notation of \ref{T:funcan} (c), for a suitable index $r$, we have $\sg (\cI_r)=\emptyset$. This statement is true because, by using the   results of 
\ref{V:am2}, \ref{V:apli} and \ref{V:cone} and  applying  the inductive hypothesis,  one sees that the functions $h_0, h_1, \ldots $  really take values in a well-ordered subset of the mentioned set $\Lambda ^{(d)}$ and they are stricly decreasing. 
 Details can be seen in \cite[4.5]{T2}.
\end{voi}
\begin{voi}
\label{V:dem5}
So far we have assumed that $\cI$ is a $d$-dimensional basic object. Suppose now 
$\cI=(M,W,(I_1,b_1), \ldots , (I_n,b_n),E)$
 is $d$-dimensional multi-ideal with, possibly, $n>1$. To finish the inductive construction, we must attach  suitable resolution functions  to $\cI$.

According to Remark \ref{R:boj}, there is a basic object $B=B_{\cI}=(M,W,(J,c),E)$ that  is equivalent to $\cI$.  If 
$h_{B,0}, \ldots, h_{B,r}$ are the algorithmic resolution functions for $B$ constructed in the previous paragraphs then, by the equivalence
$\cI \sim B$, they are also resolution functions for $\cI$. 

Had we chosen a second basic object $B'=(M,W,(J'c'),E)$, $B' \sim \cI$ then,  by transitivity of equivalence, $B \sim B'$. Since the algorithmic resolution functions $h_{B,j}$ and $h_{B',j}$ for $B$ and $B'$ respectively satisfy condition (b) of \ref{T:funcan},  it follows that 
$h_{B,j} = h_{B',j}$ for all values of $j$. 

Thus if we set $h_j:=h_{B,j}$, $j=1, \ldots, r$, where $B$ is any basic object equivalent to $\cI$, we have well-defined resolution functions attached to $\cI$. The functions $h_j$ 
satisfy conditions (a), (b), (c) of \ref{T:funcan} because 
the functions  $h_{B,j}$  satisfy them.   

So, we have  defined algorithmic resolution functions  as desired for any multi-ideal of dimension $d$, completing the inductive step and thus proving Theorem \ref{T:funcan}. 
\end{voi}

\section{Multi-ideals over artinian rings}
\label{S:Ap}

Here we explain how, using techniques similar to those of \cite{T1},  the algorithm of Section \ref{S:C} can be partially extended to multi-ideals over suitable artinian rings.
 In Theorem \ref{T:aut} we will see that, in the inductive situation,  the present algorithm behaves better than the one we used in \cite{T1}.

\begin{voi}
\label{V:ap2}
We keep the notation and terminology introduced in \ref{V:m1}. In addition,   
 we use the symbol ${\mathcal A}$ to denote the collection of artinian local rings $(A,M)$ whose residue field 
 has characteristic zero, and we write  $S:=  \Spec (A)$.

 A smooth $A$-scheme is a scheme $W$ together with a 
 a smooth morphism $p:W \to S$. Its only fiber will be  denoted by $W^{(0)}$, 
which  is an algebraic variety smooth over the field $A/M$. If $I \subset {\cO}_{W}$ is a $W$-ideal, 
 then $I^{(0)}:=I{\cO}_{W^{(0)}}$ is called the fiber of $I$. 

Let  $R={\cO}_{W,w}$,
${\bar R}={\cO}_{W^{(0)},w}$ (the local ring of the fiber at $w$, which is regular). A system of elements $a_1, \ldots, a_n$ in $R$ is called a regular system of parameters of $R$ relative to $p$, or simply  an $A$-regular system of parameters,  if the induced elements  $a_1^{(0)}, \ldots,  a_n^{(0)}$ in ${\bar R}$ form a regular system of parameters in the usual sense. 

Elements of an $A$-regular system of parameters necessarily form  a regular sequence in the local ring ${\cO}_{W,w}$ (see  \cite[11.2]{T1}). 

A hypersurface on $W$ over $S$, or an $S$- (or $A$-) hypersurface, is a positive Cartier divisor $H$, flat over $S$, inducing over the fiber $W^{(0)}$  a regular codimension one subscheme $H^{(0)}$.

We may define the notions of {\it a system $E$  of $S$-hypersurfaces of $W$ with  normal crossings (relative to $S$)} and 
{\it a subscheme $C$ of $W$  having normal crossings with $E$ (or being  transversal to $E$) over $S$, or relative to $S$}, as in \ref{V:m2}, but now working with 
{\it $A$-regular system of parameters}, rather than the usual regular systems. See \cite[3.4 and 6.1]{T1} for details.

If $C$ has normal crossings with $E$ relative to $S$, then  the induced projection $C \to S$ is smooth and  the
blowing-up $W_1$ of $W$ with center $C$ is also $S$-smooth. Moreover, the blowing-up of a smooth $A$-scheme $W$ with an $A$-smooth center $C$ is again smooth over $S$ (see \cite[11.5]{T1}). 

\end{voi}

\begin{voi}
\label{V:ap3}   If  $W$ is a smooth $S$-scheme, $I$ is a $W$-ideal, and 
 $C \subset W$  is an irreducible subscheme of $W$, smooth over $S$ 
 defined by the $W$-ideal
$J \subset {\cO}_W$,  
we shall say that the order of $I$ along
$C$  is $\geq m$,
written
$\nu (I,C) \geq m$, if $  I \subseteq J ^m     $. We write
$\nu (I,C) = m$ if $m$ is the largest integer such that $\nu (I,C) \geq m$.

If $w$ is a point of $C$, we  say that the order of $I$ at $w$ along $C$ is $\ge m$ (written $\nu_w (I,C) \ge m$) if $I_w \subseteq {J_w}^m$ 
    in the local ring ${\cO}_{W,w}$.

Let  $x_1, \ldots,  x_q$ be an $A$ -regular system of parameters of  ${\cO}_{W,w}$, so that $J_w$ is defined by the ideal  $(x_s, \ldots, x_q)$ for an appropiate index $s \ge 1$.  The completion 
$R^{\star}$ of ${\cO}_{W,w}$  with respect to $(x_1, \ldots, x_q)$ is isomorphic to a power series ring $R'[[x_1, \ldots, x_q]]$, 
 for a suitable ring $R'$. 
Then we have: 
$\nu_w (I,C) \ge m$ if and only if each $f \in I_w \subseteq R^{\star}$ is in $(x_s, \ldots, x_q)^m$. That is, when $f$, regarded  as an element of the completion 
$R^*=R'[[x_1, \ldots, x_q]]$,  is 
a power series in
$x_s, \ldots, x_q$, with coefficients in 
$ R'[[x_1, \ldots, x_{s-1}]]$ of order $\ge m$ (see \cite[3.8]{T1}).

We have $\nu(I,C) \ge m$ if and only if  $\nu _y(I,C) \ge m$, where $y$ is the generic point of $C$. Also, $\nu (I,C) \ge m$ if and only if 
$\nu _z (I,C) \ge m$ for all closed points  $z$ in a dense open subset of $C$.

\end{voi}

\begin{voi}
\label{V:moal}
 (a) If $W$ is a smooth $S$-scheme, an $A$-pair is an ordered pair $(I,b)$, where $I$ is a $W$-ideal and   $b$ a positive integer. Note that in \cite[3.4]{T1} the terminology is different.  A finite sequence of $A$-pairs is called an $A$-{\it multipair}.

(b) A closed subscheme $C$ of $W$ is a {\it permissible center} for the $A$-pair $(I,b)$ if $C$ is smooth over $S$, and if for every irreducible component $D$ of $C$ we have $\nu(I,D) \ge b $.

 Looking at fibers, we always have $\nu(I^{(0)},D^{(0)}) \ge \nu(I,d)$. 
 So, if $C$ is a permissible center for the 
 $A$-pair 
 $(W \to S,b)$, then $C^{(0)}$  is a
  permissible center for its fiber $(W ^{(0},b)$. 

(c)  Given  an $A$-pair  $(W \to S,b)$, the notions of {\it total, controlled} and {\it proper} transforms of a $W$-ideal $I$, when we blow up a permissible center of $W$, may be defined as in \ref{V:par} (ii) (see also \cite[4.5]{T1}). The operations of controlled and total transforms are compatible with that of taking fiber.
\end{voi}

\begin{voi}
\label{V:ap4}
{\it A-multi-ideals}.  Let    $A \in \cA$ and  $S =\Spec (A)$. An $A$-multi-ideal 
 is a system 
$$(1) \qquad \cI=({p:M \to S}, W, (I_1,b_1), \ldots, (I_n,b_n),E) \, , $$
where $p$ is a smooth morphism, $E=(H_1, \ldots,H_m)$ is a sequence of $A$-hypersurfaces of $M$ with normal crossings relative to $S$, $W$ is a closed subscheme of $M$, transversal to $E$ (hence automatically $S$-smooth), and $(I_i,b_i)$ is an $A$-pair of all $i$. 

If, in (1),  $n=1$ then $\cI$ is called an  $A$-marked ideal or an $A$-basic object. 

There is a natural notion of {\it fiber} ${\cI}^{(0)}$, which is a multi-ideal over $k$, the residue field of $A$. 
The $A$-multi-ideal $\cI$ is {\it nonzero} if its fiber $\cI ^{(0)}$ is nonzero (\ref{D:multi-ideal}).  

 If an index 
$v \in \{  1, \ldots, n         \}$ is fixed,  then 
a closed subscheme $C$ of $W$ is called a  {\it v-permissible center} 
 for $\cI$  if the following conditions are satisfied: (a) $C$ has normal crossings with $E$ relative to $A$, (b) $C$ is a permissible center for each of the $A$-pairs $(I_i,b_i)$, $i = 1, \ldots, n$ (\ref{V:moal}),  and  (c) 
 for each irreducible component $D$ of $C$, $\nu({I_v}^{(0)},D^{(0)}) = \nu(I_v,D) \ge b_v   $.

 Such a center $C$ is automatically smooth over $S$  \cite[11.2]{T1}. 

 The {\it transform} of an $A$-multi-ideal $\cI$ with a $v$-permissible center $C$, denoted by $\uT (\cI,C)$, is the multi-ideal  
 $(M \to S, W, (I_1[1],b_1), \ldots, (I_1[n],b_n),E')$
where, for each $i$, $(I_i[1],b_i)$ is the controlled transform of the $A$-pair $(I_i[1],b_i)$ and $E'$ consists of the strict transforms of the $A$-hypersurfaces in $E$ and the exceptional divisor,    which is again an $A$-hypersurface (see 
\cite[3.11 and 3.12]{T1}).

 A subscheme $C$ of $W$ is a  {\it permissible center }
 for $\cI$ if there is an index $v \in \{   1,\ldots,n      \}$ such that $C$ if
   {\it v-permissible center } for $\cI$. 

 A sequence of $A$-multi-ideals  
 $ \cI=  \cI_0   \leftarrow \cdots \leftarrow \cI_r   $, 
 where each arrow represents a permissible transformation, is called an $A$-permissible sequence. Such a sequence is an {\it equiresolution} if  $\sg ({\cI _r}^{(0}) = \emptyset$.
\end{voi}

\begin{voi}
\label{V:no} In \ref{R:boj}, we associated a basic object $B_{\cI}$ to a given a multi-ideal $\cI$ (over a field), and  
 $\cI$ and  $B_{\cI}$ were equivalent. If now $\cI$ is an $A$-multi-ideal, $A \in \cA$ (\ref{V:ap2}), we may define an associated $A$-basic object ${\cB}_{\cI}$ exactly as in \ref{R:boj}. We may   also adapt the definition of equivalence to the context of $A$-multi-ideals.  However, we can no longer say that $\cI$ is equivalent to ${\cB}_{\cI}$, as shown by the next example.
\end{voi}

\begin{exa}
\label{E:noe}  Consider the $A$-multi-ideal  $\cI = (M \to S,W, (I_1,b_1),(I_2,b_2), \emptyset)$  where $S= \Spec (A)$, 
  $A= k[\epsilon]=k[T]/(T^2)$ ($k$ is the field of complex numbers),
 $M=W={\mathbb A}^{1}_A= \Spec (R)$,  $R$ being the polynomial ring $A[x]$, $ I_1=(x^3)R$, $I_2=(\epsilon x+x^3)R$, $ b_1 = 3$, and  $b_2 =2$. Then, the associated $A$-basic object is 
$\cB=({M \to S},W, (x^6+3 \epsilon x^7 + x^9)R,\emptyset)=(M \to S,W, (x^6)R,\emptyset) $.  If $C$ is the subscheme of $W$ defined by $x$, i.e., $I(C)=(x)R$, then $C$ is a $\cB$-center but not an $\cI$-center. Indeed, 
$(I_2,b_2)=((\epsilon x + x^3),2)$ and so $\nu (I_2,C)=1$, hence  $\nu (I_2,C) < b_2$. So, $C$ is not a not  a permissible center for $\cI$. Consequently, $\cI$ and $\cB$ are not equivalent.
\end{exa}
 
 \begin{voi}
 \label{V:ap5} 
 (a) {\it The relative operator $\Delta$}. 
In the context of $A$-multi-ideals, $A \in \cA$, it is possible to define an operation 
$\Delta$ relative to the base $S=\Spec(A)$. Indeed, given 
 a coherent sheaf of ideals $I$ over a scheme $W$, smooth over $S$, of relative dimension $d$, define 
${\Delta}(I/S):=I + {\cF}_{d-1}(\Omega _{Y/S})$, where $Y$ is the closed subscheme of $W$ defined  by $I$ and 
${\cF} _{d-1}$ denotes the ${(d-1)}$-Fitting ideal. Then one defines 
${\Delta}^{(j)}(I/S)$ recursively:  
${\Delta}^{(0)}(I/S):=I$ and, if $j\ge1$, then 
${\Delta}^{(j)}(I/S)= {\Delta} ({\Delta}^{(j-1)}(I/S))$.

The sheaf    ${\Delta}^{(j)}(I/S)$          has the following property. Assume  $w$ is a closed point of $W$ and  $x_1, \ldots , x_d$ is an $A$-regular system of parameters of ${\cO}_{W,w}$. Consider the completion $R$ of  ${\cO}_{W,w}$ with respect to these parameters, which  is isomorphic to a power series ring $A'[[x_1, \ldots,x_d]]$, for a suitable artinian ring $A'$. Then 
${\Delta}^{(j)}(I/S)R$ is the ideal of $R$ generated by elements of $IR$ as well as their partial derivatives, up to order $j$ (see \cite[3.9]{T1}).

  \bigskip

 (b) {\it The coefficient multi-ideal}.     If    $\cI = (M \to S, W,(I,b),E)$ is       an $A$-marked ideal,  
  we define its associated  {\it coefficient multi-ideal}  $\cC(I/S)$ relative to $A$, or to $S=\Spec (A)$,   as we did in the case of basic objects over a field (\ref{V:ut1}). Namely, 
$$  \cC(I/S) = (M \to S, W, (I,b), ({\Delta}^{(1)}(I/S)),1), \ldots, ({\Delta}^{(b-1)}(I/S)), 1), E) \, . $$
Then, taking fibers, we have:  $    {\cC(I/S)}^{(0)}     =     \cC(I^{(0)})     $.
  \end{voi}
 \begin{voi} 
\label{V:ap6}
{\it Adapted  hypersurfaces and nice objects.} If 
 $\cI={(M \to S,W, (I,b),E)}$
 is  an $A$-marked ideal,  
the notions of {\it coefficient $S$-multi-ideal $(C(I/S)$} and   {\it adapted $S$-hypersurface}  
 are defined as in \ref{V:ut1} and \ref{V:ut3}. The only difference is that now we work with the relative sheaves ${\Delta}^{(j)}(I/S)$. We say that $\cI$ is {\it $A$-nice} if it admits an $A$-adapted hypersurface. If $\cI$ is  $A$-nice,  $Z$ is an adapted $S$-hypersurface and the restriction ${\cI}(C(I/S)|Z ={\cI}_Z$ is nonzero, then we say that $\cI _Z$ is the 
  {\it inductive  multi-ideal induced by $\cI$ on $Z$}..

\smallskip

Assume $\cI$ is a nice $A$-marked ideal and $Z$ a  hypersurface such that  $\cI _Z$  is nonzero. 
 Let  $C$ be a closed subscheme of $ Z \subset W$. As we shall see  in Theorem \ref{T:aut},  if $C$   is a permissible center ${\cC(\cI)}_Z$,  then $C$ is also a permissible center for $\cI$. But let us begin  with a more basic proposition.
 \end{voi}
\begin{pro}
\label{P:indi}
Assume that 
$\cI = (M \to S, W, (I,b), E)$ is an $A$-marked ideal with adapted $A$-hypersurface $Z$, such that the inductive $A$-multi-ideal ${\cI} _Z$ is nonzero. Let $C $ be a closed subscheme of $Z$ (hence  also of $W$) with normal crossings with respect  to $E$ (relative to $S$), such that 
$$ (1)  \quad\nu ({\Delta}^{(i)}(I/S)|Z,C))\ge b-i, ~ i=0, \ldots, b-1\, .$$ 
Then, $C$ is a permissible center for ${\cI}$.
\end{pro}

 \begin{proof} 
 Let us verify first that  $\nu(I,C) \ge b$.  It suffices to show:  
  $$(2) \qquad \nu_y(I,C) \ge b, ~ \forall y \in C, ~y ~ {\mathrm {closed}} $$
(see \ref{V:ap3}). To prove the inequality (2), take a closed point $y$ of $C$ and an $A$-regular system of parameters 
  $(x_1,\ldots, x_d)$ of ${\cO}_{W,y}$ such that, writing $z=x_d$, $z$ defines $C$ near $y$   and $I(C)_y=(x_s, \ldots, x_d){\cO}_{W,y}$, for a suitable index $s$. Consider the completion $R$ of ${\cO}_{W,y}$, with respect to these parameters. We have 
  $R=A'[[x_1, \ldots,x_{d-1},z]]$ (for a suitable ring $A'$, see \cite[11.6]{T1}). Then the elements $x_1, \ldots, x_{d-1}$ induce an $A$-regular  system of parameters of 
${\cO}_{Z,y}={\cO}_{W,y}/ (z)$ and the  
 completion $\bar{R}$  
  of ${\cO}_{Z,y}$ with respect to these 
  may be written as  
  $\bar{R}=A'[[x_1, \ldots, x_{d-1},z]]/(z)=A'[[x_1, \ldots, x_{d-1}]]$. To verify (2) it suffices  to show that for any 
  $  g \in {I}_y  \subset A'[[x_1, \ldots, x_{d-1},z]]       $, if we write $g$ as a power series in $x_1, \ldots, x_{d-1},z$ with coefficients in $A'$, then
$g \in (x_s, \ldots, x_d)^{b}$. 
 Let the  expression of  $g$ as a power series be 
    $g=a_0({\bf x})+a_1z({\bf x})+\cdots + a_{b_i -1}({\bf x})z^{b_i-1}+a_b({\bf x})z^{b_i}+\cdots $, where ${\bf x}=(x_1, \ldots, x_{d-1})$. 
 Then  we have:
$$a_q ({\bf x})=  (1/q!)    {(\partial ^{q} g/ \partial z ^q)}(x_1, \ldots, x_{d-1},0) \in \Delta ^{(q)}(I) {\bar R}   \, .              $$
By the assumed inequalities (1),  if $q=0, \ldots, b-1$, $a_q({\bf x}) \in (x_s, \ldots, x_{d-1}){\bar R} ^{b-q}$. This implies that  $g \in (x_s, \ldots, x_d)^{b}R$, hence  the inequality (2) is verified.

Furthermore, the fact  that $Z$ is an adapted hypersurface for $\cI$ implies that $Z^{(0)}$ is an adapted hypersurface for the fiber $\cI ^{(0)}$, hence 
${\cI}^{(0)}$ is nice and therefore good (\ref{V:ut3}). Consequently,
$b=\nu(I ^{(0)},C^{(0)})  \ge \nu(\cI ,C) \ge b $ and thus  $b=\nu(I ^{(0)},C^{(0)})  = \nu(\cI ,C)$. Since, by assumption, $C$ has normal crossings with $E$, $C$ is a permissible $A$-center for $\cI$. 
 \end{proof}

 \begin{thm}
 \label{T:aut} Let $\cI$ be a nice  $A$-basic object with adapted hypersurface $Z$, such that the inductive $A$-multi-ideal $\cI _Z$ is nonzero (this happens when the codimension of $\sg (\cI)$ is $ >1)$. 
  If $C$ is a permissible $\cI_Z$-center, then $C$ is a permissible $\cI$-center.  
  \end{thm}
 \begin{proof} By hypotehsis, $C$ has normal crossings with $E$ relative to $S$,  and since we also have the inequalities 
$\nu ({\Delta}^{(i)} (I/S))|Z,C)\ge b-i$, $i= 0, \ldots, b-1$, we may apply Proposition \ref{P:indi}.
  \end{proof}

On the contrary, if $C$ is $B$-permissible center  it does not necessarily follows that   $C$ is ${\cB}_Z$-permissible  center, as shown by the following example.

\begin{exa}
 \label{E:exma}  Consider  the $A$-basic object ${\cB}=(M \to S,W,  (I,2), \emptyset)$,  
 where  
  $A$ is as in Example  \ref{E:noe}, $S= \Spec (A)$,  
 $M=W=\Spec (A[x,y])$, $Z$ is the subscheme of $W$ defined by the ideal $(z)A[x,y]$ (hence  $Z=\Spec (A[x])$), and 
 $I=(z^2 + \epsilon x^2, z^3 +x^3) A[x,y]$. 

 Here  $C$  is a $\cB$-center because 
$  \nu (I,C)=\nu(I^{(0)},C^{(0)})=2 $. 
 But $C$ is not a $\cB _Z $-center. Indeed, 
$ \Delta ^{(0)}(I/S)|Z=  (\epsilon x^2,x^3) $,  $ \Delta ^{(1)}(I/S)|Z=  (\epsilon x,x^2) $; hence  
$${\cB} _Z=(M \to S,Z, (J_1   ,2), ( J_2  ,1), \emptyset)\, {\mathrm {with}} \,  J_1=(\epsilon x^2,x^3)A[x]  ~  {\mathrm{and}} \, J_2=  (\epsilon x, x^2)A[x] \, .  $$ 
  Then $\nu(J_i,C) \not= \nu(J_i ^{(0)},C^{(0)})$,  $i=1,2$, proving that  $C$ is not a ${\cB}_Z$-center.
\end{exa}

\providecommand{\bysame}{\leavevmode\hbox to3em{\hrulefill}\thinspace}

\end{document}